\documentclass[12pt]{amsart}
\usepackage[abbrev]{amsrefs}
\usepackage{xcolor}
\usepackage{mathtools}
\usepackage{amssymb}
\usepackage{mathrsfs}
\usepackage{dsfont}
\usepackage{esint}
\usepackage{amsthm}
\usepackage{longtable}
\usepackage{hyperref}
\hypersetup{
  colorlinks   = true, %Colours links instead of boxes
  urlcolor     = blue, %Colour for external hyperlinks
  linkcolor    = blue, %Colour of internal links
  citecolor   = red %Colour of citations
}

\theoremstyle{plain}
 \newtheorem{theorem}{Theorem}[section]
 \newtheorem*{nonum-theorem}{}

 \newtheorem{lemma}[theorem]{Lemma}
 
 \newtheorem{proposition}[theorem]{Proposition}

\newtheoremstyle{mystyle}%   
    {}%b                     
    {}%                      
    {\itshape}%              
    {}%                      
    {\bfseries}%             
    {.}%                     
    { }%                     
    {\thmnote{#3}}
 \theoremstyle{mystyle}
\newtheorem{mythm}{}

\theoremstyle{definition}
 \newtheorem{definition}[theorem]{Definition}
 
 \newtheorem*{ack}{Acknowledgment}
\theoremstyle{remark}
 \newtheorem{remark}[theorem]{Remark}

\DeclareMathOperator{\mk}{M\MRkern K}
\newcommand{\MRkern}{%
  \mkern-5.8mu
  \mathchoice{}{}{\mkern0.2mu}{\mkern0.5mu}%
}

\def\R{\mathbb{R}}
\def\N{\mathbb{N}}
\renewcommand{\S}{\mathbb{S}}
\def\rad{R}
\DeclareMathOperator*{\esssup}{ess\ sup}

\DeclareMathOperator{\spt}{spt}

\DeclareMathOperator{\dist}{\mathrm{d}}
\numberwithin{equation}{section}
\newcommand{\barypower}{\kappa}
\newcommand{\norm}{a}
\newcommand{\rast}{b}
\newcommand{\uvariable}{u}
\AtBeginDocument{%
   \def\MR#1{}
}
\begin{document}
\title[]{Sliced optimal transport: is it a suitable replacement?}
\author[]{Jun Kitagawa}
\address{Department of Mathematics, Michigan State University, East Lansing, MI 48824}
\email{kitagawa@math.msu.edu}
\author[]{Asuka Takatsu}
\address{Department of Mathematical Sciences, Tokyo Metropolitan University, Tokyo {192-0397}, Japan \&
 RIKEN Center for Advanced Intelligence Project (AIP), Tokyo Japan.}
\email{asuka@tmu.ac.jp}
\date{\today}
%Radon Tranform
%%%%%%%%%%%%%%%%%%%%%%%%%%%%%%%%%%%%%%%%
%%
\vspace{-10pt}
\begin{abstract}
We introduce a one-parameter family of metrics on 
the space of Borel probability measures on Euclidean space with finite $p$th moment for $1\leq p <\infty$, 
called the \emph{sliced Monge--Kantorovich metrics}, which include the sliced Wasserstein and max-sliced Wasserstein metrics. 
We then show that these are complete, separable metric spaces that are topologically equivalent to the classical Monge--Kantorovich metrics and these metrics have a dual representation. However, we also prove these sliced metrics are
 \emph{not} bi-Lipschitz equivalent to the classical ones in most cases, and also the spaces are (except for an endpoint case) \emph{not} geodesic. The completeness, duality, and non-geodesicness are new even in the sliced and max-sliced Wasserstein cases, and non bi-Lipschitz equivalence is only known for a few specific cases. In particular this indicates that sliced and max-sliced Wasserstein metrics are not suitable direct replacements for the classical Monge--Kantorovich metrics in problems where the specific metric or geodesic structure are critical.
\end{abstract}
%%%%%%%%%%%%%%%%%%%%%%%%%%%%%%%%%%%%%%%%%%%%
\maketitle
%%%%%%
% \tableofcontents
%%%%%%%%%%%%%%%%%%%%%%%%%%%%%%%%%%%%%%%%%%%%%%%%%%
\section{Introduction}
For $n\in \N$, we denote by $\mathcal{P}(\R^n)$ the space of all Borel probability measures on~$\R^n$, and
for $1\leq p<\infty$ we write $\mathcal{P}_p(\R^n)$ for the subset of~$\mathcal{P}(\R^n)$ consisting of measures with finite $p$th moment. 
Recall for any $m$, $n\in \N$, measure $\mu\in \mathcal{P}(\R^m)$, and  Borel map $T$ defined $\mu$-a.e. from $\R^m$ to~$\R^n$, 
the \emph{pushforward measure} $T_{\sharp} \mu \in \mathcal{P}(\R^n)$ is defined by 
\[
T_{\sharp} \mu (A):=\mu(T^{-1}(A))
\]
for any Borel set $A\subset \R^n$. 
We then denote by $\mk_p^{\R^n}$, the well-known \emph{$p$-Monge--Kantorovich metric} on $\mathcal{P}_p(\R^n)$ from optimal transport theory.
That is, if $\pi_i: \R^n\times \R^n\to \R^n$ is the projection map onto the $i$th coordinate for $i=1,2$,
then for $\mu, \nu\in \mathcal{P}_p(\R^n)$,
\begin{align}
\begin{split}\label{coupling}
 \Pi(\mu,\nu):&=\{\gamma\in \mathcal{P}(\R^n\times \R^n)\mid \pi_1{}_\sharp\gamma=\mu,\ \pi_2{}_\sharp\gamma=\nu\},\\
 \mk_p^{\R^n}(\mu, \nu):&
 =\inf_{\gamma\in \Pi(\mu, \nu)}\left(\int_{\R^n\times \R^n} \lvert x-y\rvert^pd\gamma(x, y)\right)^{\frac1p}.
\end{split}
\end{align}
By~\cite{Villani09}*{Theorem~4.1}, the infimum above is always attained. We will refer to such a minimizer as \emph{a $p$-optimal coupling} between $\mu$ and $\nu$.

%%%
It is by now well-known that $\mk_p^{\R^n}$ is a metric on $\mathcal{P}_p(\R^n)$ and provides a framework for a myriad of applications (see, for example,   \cite{Santambrogio15}*{Chapters 4, 7, and 8}, and \cite{Galichon16}). 
However, even over $\R^n$, Monge--Kantorovich metrics are expensive to compute and there is great interest in alternative metric structures on $\mathcal{P}_p(\R^n)$. 

In this paper we introduce a one-parameter family of metrics on $\mathcal{P}_p(\R^n)$ for $n\in \mathbb{N}$ based on optimal transport, 
and show various geometric properties of this family. To introduce this family, denote by $\sigma_{n-1}$ the Riemannian volume measure on the $(n-1)$-dimensional standard unit sphere~$\mathbb{S}^{n-1}$, normalized to be a probability measure.
Also for $\omega\in \S^{n-1}$, we define the map $\rad^\omega: \R^n\to \R$ for $x\in \R^n$ by
\[
 \rad^\omega(x):=\langle x,\omega\rangle.
\]
\begin{definition}\label{SMK}
For $1\leq p < \infty$ and $1\leq q \leq \infty$, we define for any $\mu, \nu\in \mathcal{P}_p(\R^n)$,
\begin{align*}
 \mk_{p, q}(\mu, \nu):&
 =\left\| \mk_p^{\R}(\rad^\bullet_\sharp\mu, \rad^\bullet_\sharp\nu)\right\|_{L^q(\sigma_{n-1})} \\
 &=\begin{cases} 
 \displaystyle
 \left(\int_{\S^{n-1}}\mk_p^{\R}(\rad^\omega_\sharp\mu, \rad^\omega_\sharp\nu)^qd\sigma_{n-1}(\omega)\right)^{\frac{1}{q}},&\text{if\ } 1\leq q<\infty,\\
 \displaystyle
 \esssup_{\omega\in\S^{n-1}}\mk_p^{\R}(\rad^\omega_\sharp\mu, \rad^\omega_\sharp\nu),&\text{if\ } q=\infty.
\end{cases}
\end{align*}
We call $\mk_{p, q}$ the \emph{sliced $(p, q)$-Monge--Kantorovich metric},
and will generically refer to \emph{sliced Monge--Kantorovich metrics}.
\end{definition}
We note that $\rad^\omega_\sharp \mu$ is the Radon transform of $\mu$, and the cases $\mk_{p, p}$ are known as the so-called \emph{sliced Wasserstein metrics}, while $\mk_{p, \infty}$ are often referred to as 
the \emph{max-sliced Wasserstein metrics}. 
See~\cites{BayraktarGuo21, sliced-original, mSWD} 
and the references therein for related works.
%%%
To state our main results on the sliced Monge--Kantorovich metrics, 
we first fix some notation.
%%%
For a metric space $(X,\dist_X)$,
let $C_b(X)$ denote the space of bounded continuous functions on $X$.
For $1\leq p<\infty$ and $1\leq r' \leq \infty$, 
define 
\begin{align*}
\mathcal{A}_p
&:=\left\{
(\Phi_\bullet, \Psi_\bullet)\in C(\S^{n-1}; C_b(\R))^2
\Biggm|
\begin{tabular}{l}
$-\Phi_\omega(t)-\Psi_\omega(s) \leq |t-s|^p$ \\
for all $t$, $s\in \mathbb{R}$ and $\omega\in \mathbb{S}^{n-1}$
\end{tabular}
\right\},\\
%%%%
\mathcal{Z}_{r'}
&:=
\left\{ \zeta\in C(\S^{n-1}) \bigm| \lVert \zeta\rVert_{L^{r'}(\sigma_{n-1})}\leq 1,\ \zeta>0
\right\}.
\end{align*}

%%%%%%%%%%%%%%%%%%%%%%%%%
It is easily seen that when $n=1$, one has $\mk_{p, q}=\mk_p^\R$ for all $p$ and $q$, hence the only interesting cases which we will treat for $\mk_{p, q}$ are when $n\geq 2$.
\begin{mythm}[Main Theorem]\label{thm: main sliced}
Let $1\leq p<\infty$, $1\leq q\leq \infty$ and $n\in\mathbb{N}$. Then:
\begin{enumerate}
\setlength{\leftskip}{-15pt}
\item\label{thm: sliced complete} 
$(\mathcal{P}_p(\mathbb{R}^n), \mk_{p,q})$ is a complete, separable metric space.
\item\label{thm: sliced top equiv} 
$(\mathcal{P}_p(\mathbb{R}^n), \mk_{p,q})$ is  topologically equivalent to $(\mathcal{P}_p(\mathbb{R}^n), \mk_p^{\mathbb{R}^n})$.
\item\label{thm: sliced not metric equiv} 
$(\mathcal{P}_p(\mathbb{R}^n), \mk_{p,q})$ and 
$(\mathcal{P}_p(\mathbb{R}^n), \mk_p^{\mathbb{R}^n})$ are  \emph{not} bi-Lipschitz equivalent when $n\geq 2$ in any of the following cases:
\begin{itemize}
\setlength{\itemindent}{-15pt} 
    \item $q\neq \infty$ and $1\leq p<\infty$.
    \item $q=\infty$, $n=2$, and $p=1$.
    \item $q=\infty$, $n\geq 3$, and $1\leq p<n/2$.
\end{itemize}
\item\label{thm: sliced not geodesic} 
$(\mathcal{P}_p(\mathbb{R}^n), \mk_{p,q})$ is a geodesic space 
if and only if $p=1$ or $n=1$.
\item \label{thm: barycenters}
Assume $p\leq q$.
Fix any $K\in \N$ with $K\geq 2$, $\{\mu_k\}_{k=1}^K\subset\mathcal{P}_p(\R^n)^K$, 
$\lambda_1,\ldots,\lambda_K>0$ such that $\sum_{k=1}^K\lambda_k=1$, and $\barypower\geq 0$. 
Then there exists a minimizer of 
\begin{align*}
    \nu\mapsto \sum_{k=1}^K\lambda_k \mk_{p, q}(\mu_k, \nu)^{\barypower}
\end{align*}
in~$\mathcal{P}_p(\R^n)$.
\item\label{thm: sliced duality} 
Assume $p\leq q$.
Set $r:=q/p$, and denote by $r'$ the H\"older conjugate of $r$.
Then for $\mu$, $\nu\in \mathcal{P}_p(\R^n)$, 
we have
\begin{align*}
\mk_{p, q}(\mu, \nu)^p =
\sup
\left\{
-\int_{\S^{n-1}}\zeta\left(\int_{\R}\Phi_\bullet dR^\bullet_\sharp\mu
+\int_{\R}\Psi_\bullet dR^\bullet_\sharp\nu\right) d\sigma_{n-1}
\Biggm|
(\Phi,\Psi)\in \mathcal{A}_p,\ 
\zeta \in \mathcal{Z}_{r'}
\right\}.
%%%%
\end{align*}
\end{enumerate}
\end{mythm}
Regarding \nameref{thm: main sliced}~\eqref{thm: sliced not geodesic} above, recall the following definition.
\begin{definition}
Let $(X,\dist_X)$ be a metric space.
A curve $\rho:[0,1]\to X$ is called a
 \emph{minimal geodesic} if 
\begin{equation}\label{geodineq}
\dist_X(\rho(\tau_1), \rho(\tau_2))\leq |\tau_1-\tau_2|\dist_X(\rho(0),\rho(1))
\end{equation}
holds for any $\tau_1$, $\tau_2\in [0,1]$. 
A metric space $(X,\dist_X)$ is said to be \emph{geodesic} 
if any two points in $X$ can be joined by a minimal geodesic.
\end{definition}
Note that due to the triangle inequality, 
 equality actually holds in \eqref{geodineq}
for a minimal geodesic.
\begin{remark}
    The proof of \nameref{thm: main sliced}~\eqref{thm: sliced not metric equiv} actually shows a stronger result holds. Namely, the bi-Lipschitz equivalence also fails for measures whose supports are contained in a fixed bounded set, even with a constant dependent on this fixed set.

    Also, the assumed restrictions in the case $q=\infty$ in \nameref{thm: main sliced}~\eqref{thm: sliced not metric equiv} are due to the proof technique used. Our proof is based on expectation bounds on the $\mk_{1, \infty}$ distance between a measure and its i.i.d. empirical samplings which are of a strongly generic nature, hence its validity is limited in the range of $p$. We strongly believe the same non bi-Lipschitz nature should hold for all values of $p$ when $q=\infty$.
\end{remark}
\begin{remark}\label{rem: p=q duality}
    Upon inspection of the proof in \nameref{thm: main sliced}~\eqref{thm: sliced duality}, it is easy to see that when $p=q$ (hence $r'=\infty$), the maximum value is attained when $\zeta\equiv 1$, thus the supremum over $\zeta$ is not needed.
\end{remark}
%%%%%%%%%%%%%%%%%
\subsection*{Literature}
There is not much in the way of systematic study of optimal transport metrics constructed using the Radon transform, but variants such as sliced and max-sliced Wasserstein are used extensively in applications. As briefly mentioned, these variants are used in various applications as computationally faster substitutes for the classical Monge--Kantorovich metric. However this substitution should be made with care, as \nameref{thm: main sliced}~\eqref{thm: sliced not metric equiv} and~\eqref{thm: sliced not geodesic} indicate that sliced and max-sliced Wasserstein metrics are \emph{not} suitable direct replacements for the classical metrics in problems where the exact metric nature or geodesic properties are of importance. We also note, although the spaces $(\mathcal{P}_1(\R^n), \mk_{1, q})$ are geodesic, the geodesics are simple linear convex combinations and do not actually reflect the underlying geometry of the base space $\R^n$ (see Proposition~\ref{geod}).

It is proved in~\cite{BayraktarGuo21}*{Theorem~2.3} that 
$(\mathcal{P}_p(\mathbb{R}^n), \mk_{p,q})$ with $p=q$ or $q=\infty$ is topologically  equivalent to $(\mathcal{P}_p(\mathbb{R}^n), \mk_p^{\mathbb{R}^n})$ for any $n\in \mathbb{N}$. 
However, it is also shown there that 
$(\mathcal{P}_p(\mathbb{R}^n), \mk_{1,1})$ and 
$(\mathcal{P}_p(\mathbb{R}^n), \mk_1^{\mathbb{R}^n})$ 
are bi-Lipschitz  equivalent to  each other if and only if $n=1$.
%%%%%%
On the other hand, it is known that 
$\mk_{p, p}\leq C\mk_p^{\mathbb{R}^n}$ for some constant $C$, 
and it is shown in \cite{Bonnotte}*{Theorem~5.1.5} that 
$\mk_p^{\mathbb{R}^n}\leq C\mk_{p, p}^{1/(n+1)}$ holds for measures with compact supports
for some constant $C$, where the constant depends on the diameter of the union of the two supports 
(see also~\cite{BayraktarGuo21}*{Remark~2.4}).

 The completeness, (non)-geodesicness, and duality in our main result are new even in the sliced and max-sliced Wasserstein cases $p=q$ and $q=\infty$. For the non-geodesicness, we explicitly give two measures not connected by a $\mk_{p, q}$-minimal geodesic. The counterexample exploits that if one has a $\mk_{p, q}$-minimal geodesic $\mu(\tau)$, then the curves $\tau\mapsto \rad^\omega\mu(\tau)$ must be $\mk_p^\R$-minimal geodesics for certain values of $\omega$. The duality result is in some sense the expected one, from applying the classical Kantorovich duality to each pair of one-dimensional transport problems. However, one must take care when making a selection of admissible dual potentials to ensure some reasonable dependence in the $\omega$ variable, we exploit Michael's continuous selection theorem (see Proposition~\ref{Michael56}) to accomplish this. We note that the metric $\mk_{2, \infty}$ is equal to the \emph{$1$-dimensional projection robust $2$-Wasserstein distance} defined in ~\cite{CuturiPaty19}. In \cite{CuturiPaty19} the authors also define what are called \emph{subspace robust $2$-Wasserstein distances}, which are shown to be bi-Lipschitz equivalent to $\mk_2^{\R^n}$ and geodesic, however these are not included in our family $\mk_{p, q}$.

Barycenters in the sliced setting have previously been considered in~\cite{BonneelRabinPeyrePfister15}. We are able to obtain existence of $\mk_{p, q}$-barycenters (i.e. minimizers as in \eqref{thm: barycenters}) by a simple compactness argument, but were unable to prove a duality result (as in \cite{AguehCarlier11}*{Proposition 2.2} for $\mk_p^{\R^n}$-barycenters). This is partly due to the fact that the Radon transform (push-forward under $\rad^\omega$) and its dual transform are ill-behaved on Banach spaces, and are only known to have nice properties on Fr\'echet spaces whose topologies are defined by a countably infinite family of semi-norms (see for example, \cites{Hertle83, Hertle84}). We also note that the Radon transforms of an $\mk_{p, q}$-barycenter are not necessarily the usual $\mk_p^\R$-barycenters of the Radon transforms of the fixed measures (as seen from the proof of \nameref{thm: main sliced}~\eqref{thm: sliced not geodesic}); this is in concordance with~\cite{BonneelRabinPeyrePfister15}*{Proposition 10 (39)}. In a forthcoming work, we will introduce an alternate family of metrics which are geodesic and for which we can show duality for the associated barycenter problems. It will turn out that the spaces $(\mathcal{P}_p(\R^n), \mk_{p,q})$ can be isometrically embedded into certain cases of these alternate spaces.

 The authors have also recently learned of an independent result~\cite{ParkSlepcev23} by Park and Slep{\v c}ev in which they have shown topological properties, completeness, and non-geodesicness of the space $(\mathcal{P}_2(\R^n), \mk_{2, 2})$. The overall scope of this paper differs from that of~\cite{ParkSlepcev23}, as they also consider topics such as tangential structure and statistical properties of $\mk_{2, 2}$, and relations to Sobolev norms, while we consider cases other than $p=q=2$, and duality.

\section{Proofs of main results}\label{sec: sliced}
 We start with a few preliminary properties before moving to the proof of the actual theorem. Throughout this paper, 
we will take $1\leq p<\infty$, $1\leq q \leq \infty$ 
and $n\in \mathbb{N}$ unless stated otherwise. First, we recall here some properties of the usual Monge--Kantorovich metrics for later use. We will write $B_r^{\R^n}(x)$ for the open ball centered at $x\in \R^n$ of radius $r>0$ with respect to the Euclidean distance.
%%%%%%%%%%%%%%%%%
\begin{theorem}[\cite{Villani09}*{Theorem~6.9, Theorem~6.18}]\label{thm: wassconv}
Then $(\mathcal{P}_p(\R^n), \mk_p^{\R^n})$ is a complete, separable metric space.
%%%%

For a sequence $(\mu_j)_{j\in \mathbb{N}}$ in $\mathcal{P}_p(\R^n)$
and $\mu \in \mathcal{P}_p(\R^n)$, 
the following are equivalent.
\begin{itemize}
\setlength{\leftskip}{-15pt}
\item
$\lim_{j\to \infty} \mk_p^{\R^n}(\mu_j,\mu)=0$.
\item
$(\mu_j)_{j\in \mathbb{N}}$ converges weakly to $\mu$ 
and 
\[
\lim_{j\to \infty} 
\int_{\R^n} \lvert x\rvert^p d\mu_j(x)=
\int_{\R^n} \lvert x\rvert^p d\mu(x).
\]
\item
$(\mu_j)_{j\in \mathbb{N}}$ converges weakly to $\mu$ 
and 
\[
\lim_{r\to\infty}\limsup_{j\to \infty} 
\int_{\R^n\setminus B_r^{\R^n}(0)} \lvert x\rvert^p d\mu_j(x)=0.
\]
\item
For any continuous function on $\phi$ on $\R^n$ with 
$|\phi|\leq C(1+\lvert \cdot\rvert^p)$ 
for some $C\in \mathbb{R}$, 
\[
\lim_{j\to \infty} 
\int_{\R^n} \phi d\mu_j=
\int_{\R^n} \phi d\mu.
\]
\end{itemize}
\end{theorem}
Next, some notation and conventions. For $1\leq i\leq n$, we write $e_i$ for the $i$th coordinate vector in $\R^n$.
We also denote by $\mathds{1}_E$ the characteristic function of a set $E$. Finally, we will write $\delta^{\R^n}_x$ to denote the delta measure at the point $x$ on $\R^n$.

We will first show continuity on~$\mathbb{S}^{n-1}$ of the function defined by 
$\omega \mapsto \mk_p^{\R}(\rad^\omega_\sharp\mu, \rad^\omega_\sharp\nu)$
for $\mu,\nu\in\mathcal{P}_p(\mathbb{R}^n)$ to ensure that $\mk_{p,q}(\mu,\nu)$ is well-defined. 
This is also proved in ~\cite{BayraktarGuo21}*{Proposition~2.3}, but we provide the proof here for completeness.
\begin{lemma}\label{prelemma}
Let $\mu$, $\nu \in \mathcal{P}_p(\mathbb{R}^n)$. 
\begin{enumerate}
\setlength{\leftskip}{-15pt}
\item\label{moment}
For $\omega \in \mathbb{S}^{n-1}$, 
$\rad^\omega_\sharp \mu \in \mathcal{P}_p(\mathbb{R})$ holds.
Moreover, the $p$th moment of $\rad^\omega_\sharp \mu$ is bounded by the $p$th moment of $\mu$.
%%%%%%
\item\label{conti}
For any sequence $(\omega_j)_{j\in \N}$ in $\mathbb{S}^{n-1}$ converging to $\omega$, 
\begin{align*}
\lim_{j\to \infty} 
\mk_p^{\R}(\rad^{\omega_j}_\sharp\mu, \rad^\omega_\sharp\mu)=0,
\end{align*}
in particular, the sequence $(\rad^{\omega_j}_\sharp\mu)_{j\in\mathbb{N}}$ weakly converges to $\rad^{\omega}_\sharp\mu$.
\item\label{twocont}
The function on $\mathbb{S}^{n-1}$ defined by 
\[
\omega \mapsto \mk_p^{\R}(\rad^\omega_\sharp\mu, \rad^\omega_\sharp\nu)
\]
is continuous.
\end{enumerate}
\end{lemma}
\begin{proof}
We calculate 
\begin{align*}
\int_{\mathbb{R}} |t|^p d \rad^\omega_\sharp \mu(t)
=
\int_{\mathbb{R}^n} |\langle x,\omega \rangle|^p d \mu(x)
\leq
\int_{\mathbb{R}^n}  |x|^p |\omega|^p d\mu(x)
=\int_{\mathbb{R}^n} |x|^p d\mu(x)
=\mk_p^{\R^n}(\delta_0^{\mathbb{R}^n}, \mu)^p,
\end{align*}
which proves assertion~\eqref{moment}.

Let $(\omega_j)_{j\in \N}$ be a convergent sequence in $\mathbb{S}^{n-1}$ with limit $\omega$.
Then for any $\phi\in C_b(\R)$, we have $\phi(\langle \cdot, \omega_j\rangle)\in C_b(\R^n)$ for each $j\in\mathbb{N}$ hence by dominated convergence,
\begin{align*}
\lim_{j\to \infty}
\int_{\R}\phi d\rad^{\omega_j}_\sharp\mu
=
\lim_{j\to \infty}
\int_{\R^n}\phi(\langle x, \omega_j\rangle)d\mu(x)
= \int_{\R^n}\phi(\langle x,\omega\rangle)d\mu(x)=\int_{\R}\phi d\rad^\omega_\sharp\mu,
\end{align*}
thus $(\rad^{\omega_j}_\sharp\mu)_{j\in\N}$ converges weakly to $\rad^\omega_\sharp\mu$.
Moreover, since $\left| \langle x, \omega_j\rangle \right|^p\leq \lvert x\rvert^p$ for $x\in \mathbb{R}^n$ and $\mu$ has finite $p$th moment, dominated convergence yields 
\begin{align*}
\lim_{j\to \infty}
\int_{\R} |t|^p d\rad^{\omega_j}_\sharp\mu(t)
=
\lim_{j\to \infty} \int_{\R^n}\left| \langle x, \omega_j\rangle \right|^p d\mu(x)
=
\int_{\R^n}\left| \langle x, \omega\rangle \right|^p d\mu(x)
=\int_{\R} |t|^p d\rad^\omega_\sharp\mu(t).
\end{align*}
Then it follows from Theorem~\ref{thm: wassconv} 
that
\begin{align*}
\lim_{j\to \infty} 
\mk_p^{\R}(\rad^{\omega_j}_\sharp\mu, \rad^\omega_\sharp\mu)=0,
\end{align*}
and assertion~\eqref{conti} follows.

Assertion~\eqref{conti} for $\mu,\nu$ together with the triangle inequality for $\mk_p^{\R}$
leads to 
\begin{align*}
\lim_{j\to \infty} \mk_p^{\R}(\rad^{\omega_j}_\sharp\mu, \rad^{\omega_j}_\sharp\nu)
=
\mk_p^{\R}(\rad^\omega_\sharp\mu, \rad^\omega_\sharp\nu),
\end{align*}
proving assertion~\eqref{twocont}.
\end{proof}
Next we give a comparison result between $\mk_p^{\R^n}$ and $\mk_{p,q}$.
%%%
For $1\leq q\leq\infty$ and $n\in \mathbb{N}$, set 
\[
M_{q,n}
:=\| \langle e_1,\cdot\rangle \|_{L^q(\sigma_{n-1})}.
\]
Using the fact that (see~\cite{Stroock94}*{(5.2.5.(ii))} for instance) for any $v_1, v_2\in \R^n$ 
\[
 \langle v_1, v_2\rangle=n\int_{\S^{n-1}}\langle v_1, \omega\rangle \langle v_2, \omega\rangle d\sigma_{n-1}(\omega),
\]
we have $M_{2,n}=n^{-1/2}$; additionally it is easy to see that $M_{\infty,n}=1$.
We also observe from the rotational invariance of $\sigma_{n-1}$ that 
\[
M_{q,n} |x|
=\| \langle x,\cdot \rangle \|_{L^q(\sigma_{n-1})}
\]
holds for any $x\in \mathbb{R}^n$.
The following comparison in Lemma~\ref{comparison} is proved for $p=q$ in \cite{Bonnotte}*{Proposition~5.1.3}. Note that a simple application of H\"{o}lder's inequality yields 
\[
p\leq p', q\leq q'
\Rightarrow
\mk_{p, q}
\leq
\mk_{p', q'},
\]
which will be used in the sequel.
\begin{lemma}\label{comparison}
For $\mu,\nu\in \mathcal{P}_p(\R^n)$,
\[
\mk_{p,q}(\mu, \nu)
\leq
M_{\max\{p,q\},n} \cdot
\mk_p^{\R^n}(\mu, \nu).
\]
%%%
Additionally, 
for any $x_0\in \mathbb{R}^n$, 
\[
\mk_{p,p}(\delta_{x_0}^{\R^n},\mu)
=
M_{p,n} \cdot
\mk_p^{\R^n}(\delta_{x_0}^{\R^n},\mu).
\]
Consequently, 
$\mathbb{R}^n$ is homothetically embedded into $(\mathcal{P}_p(\mathbb{R}^n), \mk_{p,p})$.
\end{lemma}
\begin{proof}
Let $\gamma$ be a $p$-optimal coupling between $\mu$ and $\nu$, 
then $(\rad^\omega\times \rad^\omega)_\sharp \gamma \in \Pi(\rad^\omega_\sharp \mu,\rad^\omega_\sharp \nu)$
for each $\omega\in \mathbb{S}^{n-1}$.
%%%%
We calculate
\begin{align}\label{tonelli}
\begin{split}
\mk_{p, q}(\mu, \nu)
&=\left\|\mk_p^{\R}(\rad^\bullet_\sharp\mu, \rad^\bullet_\sharp\nu)\right\|_{L^q(\sigma_{n-1})} \\
%%%%
&\leq \left\|\left(\int_{\mathbb{R}\times \mathbb{R}} |t-s|^p d(\rad^\bullet\times\rad^\bullet)_\sharp \gamma(t,s)\right)^{\frac{1}{p}}\right\|_{L^q(\sigma_{n-1})}\\
 %%%%%
&=\left\|
 \left(\int_{\mathbb{R}^n\times \mathbb{R}^n} |\langle x-y,\cdot\rangle|^p d\gamma(x,y)\right)^{\frac{1}{p}} \right\|_{L^q(\sigma_{n-1})}.
%%%%
\end{split}
\end{align}
When $q=\infty$, we easily obtain
\begin{align*}
 \left\|
 \left(\int_{\mathbb{R}^n\times \mathbb{R}^n} |\langle x-y,\cdot\rangle|^p d\gamma(x,y)\right)^{\frac{1}{p}} \right\|_{L^\infty(\sigma_{n-1})} 
 &\leq 
 \left(\int_{\mathbb{R}^n\times \mathbb{R}^n} \left\||\langle x-y,\cdot\rangle|^p\right\|_{L^\infty(\sigma_{n-1})} d\gamma(x,y)\right)^{\frac{1}{p}} \\
 &= \left(\int_{\mathbb{R}^n\times \mathbb{R}^n}| x-y|^p d\gamma(x,y)\right)^{\frac{1}{p}}\\
 &=M_{\infty,n}\cdot \mk_p^{\R^n}(\mu, \nu),
\end{align*}
where we have used $M_{\infty,n}=1$.

If $p\leq q<\infty$, 
by Minkowski's integral inequality, we find
\begin{align*}
\mk_{p, q}(\mu, \nu)
\leq &\left\|
 \left(\int_{\mathbb{R}^n\times \mathbb{R}^n} |\langle x-y,\cdot\rangle|^p d\gamma(x,y)\right)^{\frac{1}{p}} \right\|_{L^q(\sigma_{n-1})}\\
&=
 \left[ \int_{\S^{n-1}}
 \left(\int_{\R^n \times \mathbb{R}^n} |\langle x-y,\omega\rangle|^p d\gamma(x,y)\right)^{\frac{q}{p}}
 d\sigma_{n-1}(\omega)\right]^{\frac{p}{q}\cdot\frac1p}\\
&\leq \left[ \int_{\R^n \times \mathbb{R}^n} \left(\int_{\S^{n-1}}|\langle x-y,\omega\rangle|^{q}d\sigma_{n-1}(\omega) \right)^{\frac{p}{q}}d\gamma(x,y)\right]^{\frac{1}{p}}\\
%%%%
&= \left( \int_{\R^n \times \mathbb{R}^n} 
M_{q,n}^p |x-y|^p d\gamma(x,y)\right)^{\frac{1}{p}}\\
&=M_{q,n}\cdot \mk_p^{\R^n}(\mu, \nu).
\end{align*}
When $q<p$, the above with the aforementioned monotonicity of $\mk_{p, q}$ in $q$ shows 
\[
\mk_{p, q}(\mu, \nu) 
\leq
\mk_{p, p}(\mu, \nu) 
\leq
M_{p,n}\cdot \mk_p^{\R^n}(\mu, \nu).
\]

Since we have
\[
\Pi(\delta_{x_0}^{\R^n},\mu)=\{ \delta_{x_0}^{\R^n} \otimes \mu \}, 
\quad
\Pi(\rad^\omega_\sharp \delta_{x_0}^{\R^n},\rad^\omega_\sharp \mu)
=\{( \rad^\omega_\sharp \delta_{x_0}^{\R^n})\otimes ( \rad^\omega_\sharp \mu) \}=\{(\rad^\omega\times\rad^\omega)_\sharp (\delta_{x_0}^{\R^n}\otimes\mu)\}
\]
for each $x_0\in\mathbb{R}^n$ and $\omega \in \mathbb{S}^{n-1}$,
the inequality in \eqref{tonelli} becomes an equality 
and this proves the lemma.
\end{proof}
\begin{remark}\label{rem}
By Lemma~\ref{comparison}, for $\mu$, $\nu\in \mathcal{P}_p(\R^n)$, and $x_0\in \mathbb{R}^n$, 
if $\mk_p^{\R^n}(\delta_{x_0}^{\R^n},\mu)=\mk_p^{\R^n}(\delta_{x_0}^{\R^n}, \nu)$, 
then $\mk_{p, p}(\delta_{x_0}^{\R^n},\mu)=\mk_{p, p}(\delta_{x_0}^{\R^n}, \nu)$.
However, this does not hold for $p\neq q$ and $n\geq 2$ in general.
Indeed, if we take 
\[
\mu:=\delta_{e_1}^{\mathbb{R}^n}, \quad
\nu:=\frac12\left(\delta_{e_1}^{\R^n}+\delta_{e_2}^{\mathbb{R}^n}\right), \quad
\]
and $p\neq q$, then 
\begin{align*}
%%%%%
\mk_{p, q}(\delta_0^{\mathbb{R}^n},\mu)
&=\left\|\mk_p^{\R}( \delta_{0}^{\mathbb{R}},\delta_{\langle e_1,\bullet\rangle}^\R )
\right\|_{L^q(\sigma_{n-1})}
=\left\|\langle e_1,\cdot\rangle\right\|_{L^q(\sigma_{n-1})}
=M_{q,n},\\
%%%%%
\mk_{p, q}(\delta_0^{\mathbb{R}^n},\nu)
&=\left\|\mk_p^{\R}\left( \delta_{0}^{\mathbb{R}},\frac12\left(\delta_{\langle e_1,\bullet\rangle}^\R+\delta_{\langle e_2,\bullet\rangle}^\R\right) \right)\right\|_{L^q(\sigma_{n-1})}
=\left\| \left( \frac12 |\langle e_1,\cdot\rangle|^p+\frac12|\langle e_2,\cdot\rangle|^p \right)^{\frac{1}{p}}\right\|_{L^q(\sigma_{n-1})}.
\end{align*}
We find that $\mk_{p, \infty}(\delta_0^{\mathbb{R}^n},\nu )<1=M_{\infty,n}$.
If $p<q<\infty$, then it follows from strict convexity of the function $t\mapsto t^{q/p}$ on $(0,\infty)$ that
\begin{align*}
\mk_{p, q}(\delta_0^{\mathbb{R}^n},\nu )^q
&=
\int_{\mathbb{S}^{n-1}} \left( \frac12 |\langle e_1,\omega\rangle|^p+\frac12|\langle e_2,\omega\rangle|^p \right)^{\frac{q}{p}} d\sigma_{n-1}(\omega)\\
&<\int_{\mathbb{S}^{n-1}} \left( \frac12 |\langle e_1,\omega\rangle|^q +\frac12|\langle e_2,\omega\rangle|^q \right) d\sigma_{n-1}(\omega)
=M_{q,n}^q.
\end{align*}
If $1\leq q<p$, then we similarly observe from strict concavity of the function $t\mapsto t^{q/p}$ on $(0,\infty)$ that $\mk_{p, q}(\delta_0^{\mathbb{R}^n},\nu)>M_{q,n}$.
Thus we find 
\[
\mk_p^{\R^n}(\delta_0^{\mathbb{R}^n},\mu)=\mk_p^{\R^n}(\delta_0^{\mathbb{R}^n},\nu )=1, \qquad
\mk_{p, q}(\delta_0^{\mathbb{R}^n},\mu ) \neq \mk_{p, q}(\delta_0^{\mathbb{R}^n},\nu).
\]
\end{remark}
\subsection{Complete, separable, metric.}
In this subsection we will prove $(\mathcal{P}_p(\R^n), \mk_{p, q})$ is a complete, separable metric space. Before the proof, we make a quick remark which will be used a number of times.
\begin{remark}\label{rmk: mk compact}
By~\cite{AmbrosioGigliSavare08}*{Remark 5.1.5}, any sequence $(\mu_j)_{j\in \N}$ in $\mathcal{P}_p(\R^n)$ with uniformly bounded $p$th moments (or equivalently, is bounded in $(\mathcal{P}_p(\R^n), \mk_p^{\R^n})$) has a subsequence that weakly converges to some $\mu\in \mathcal{P}(\R^n)$. Then since 
   $\mk_p^{\R^n}( \delta_{0}^{\R^n},\mu_j)^p$ is the $p$th moment of $\mu_j$, by the weak lower-semicontinuity of $\mk_p^{\R^n}$ (\cite{Villani09}*{Remark 6.12}) the limiting measure $\mu$ also has finite $p$th moment.
\end{remark}
\begin{proof}[Proof of \nameref{thm: main sliced}~\eqref{thm: sliced complete}]
\ \\[-10pt]

\noindent
\textbf{(Metric):} Let $\mu, \nu\in \mathcal{P}_p(\R^n)$.
From the definition it is immediate that 
\[
\mk_{p, q}(\mu, \nu)=\mk_{p, q}(\nu, \mu)
\geq 0,
\]
and since $\mk_p^{\R}$ is a metric on $\mathcal{P}_p(\R)$, that $\mk_{p, q}(\mu, \mu)=0$.
Now suppose $\mk_{p, q}(\mu, \nu)=0$.
Then for $\sigma_{n-1}$-a.e.\ $\omega$ we have
$\mk_p^\R(\rad^\omega_\sharp\mu, \rad^\omega_\sharp\nu)=0$, 
hence $\rad^\omega_\sharp\mu=\rad^\omega_\sharp\nu$. 
Writing $\mathcal{F}_\R^{-1}$ and $\mathcal{F}_{\R^n}^{-1}$ for the inverse Fourier transforms on $\R$ and $\R^n$ respectively, a quick calculation yields that for any $r>0$ and 
$\sigma_{n-1}$-a.e.\ $\omega$,
\begin{align*}
 \mathcal{F}_{\R^n}^{-1} \mu(r\omega)&=\int_{\R^n}e^{i\langle r\omega, x\rangle} d\mu(x)
 =\int_\R e^{i r t}d \rad^\omega_\sharp\mu(t)=\mathcal{F}_\R^{-1} (\rad^\omega_\sharp\mu)(r)\\
 &=\mathcal{F}_\R^{-1} (\rad^\omega_\sharp\nu)(r)
 =\int_\R e^{i r t}d \rad^\omega_\sharp\nu(t)
 =\int_{\R^n}e^{i\langle r\omega, x\rangle} d\nu(x)=\mathcal{F}_{\R^n}^{-1} \nu(r\omega).
\end{align*}
Since a probability measure is uniquely determined by its inverse Fourier transform, (%i.e. its characteristic function,
see~\cite{Bogachev07}*{Proposition~3.8.6} for instance), we have $\mu=\nu$. 
For the triangle inequality, 
using the triangle inequality for $\mk_p^\R$ together with Minkowski's inequality, we have for $\mu_1, \mu_2, \mu_3\in \mathcal{P}_p(\R^n)$,
\begin{align*}
 \mk_{p, q}(\mu_1, \mu_3)&=\left\| \mk_p^{\R}(\rad^\bullet_\sharp\mu_1, \rad^\bullet_\sharp\mu_3)\right\|_{L^q(\sigma_{n-1})}\\
 &\leq \left\| \mk_p^{\R}(\rad^\bullet_\sharp\mu_1, \rad^\bullet_\sharp\mu_2)+\mk_p^{\R}(\rad^\bullet_\sharp\mu_2, \rad^\bullet_\sharp\mu_3)\right\|_{L^q(\sigma_{n-1})}\\
 &\leq \left\| \mk_p^{\R}(\rad^\bullet_\sharp\mu_1, \rad^\bullet_\sharp\mu_2)\right\|_{L^q(\sigma_{n-1})}+\left\|\mk_p^{\R}(\rad^\bullet_\sharp\mu_2, \rad^\bullet_\sharp\mu_3)\right\|_{L^q(\sigma_{n-1})}\\
 &=\mk_{p, q}(\mu_1, \mu_2)+\mk_{p, q}(\mu_2, \mu_3).
\end{align*}

\noindent
\textbf{(Separability):} 
Due to Theorem~\ref{thm: wassconv}, $(\mathcal{P}_p(\R^n), \mk_p^{\mathbb{R}^n})$ is separable 
and there exists a countable dense set $\mathcal{D}$ of $(\mathcal{P}_p(\R^n), \mk_p^{\mathbb{R}^n})$.
By Lemma~\ref{comparison}, 
the set $\mathcal{D}$ is also dense in $(\mathcal{P}_p(\R^n), \mk_{p,q})$ 
hence $(\mathcal{P}_p(\R^n), \mk_{p,q})$ is separable.

\noindent
\textbf{(Completeness):} 
Let $(\mu_j)_{j\in \N}$ be a Cauchy sequence in $(\mathcal{P}_p(\R^n), \mk_{p,q})$. 
We prove the completeness of $(\mathcal{P}_p(\R^n), \mk_{p,q})$ in several steps.

\noindent
{\bf Claim~1.}
There exists $\Omega_{p, q} \subset \mathbb{S}^{n-1}$ such that 
$\sigma_{n-1}(\Omega_{p, q})=1$ and $(\rad^\omega_\sharp\mu_j)_{j\in \N}$ is Cauchy in $\mk_p^\R$ for any $\omega\in \Omega_{p, q}$.\\
{\it Proof of Claim~$1$.}
If $q=\infty$, then the claim is trivial.
In the case $q\neq\infty$, for any $\varepsilon_1, \varepsilon_2>0$, there exists some $J\in \N$ such that whenever $j_1, j_2\geq J$, we have $\mk_{p, q}(\mu_{j_1}, \mu_{j_2})<\varepsilon_1\varepsilon_2$. 
It follows from Chebyshev's inequality that 
\begin{align*}
 \sigma_{n-1}\left(\{\omega\in \S^{n-1}\mid \mk_p^\R(\rad^\omega_\sharp\mu_{j_1}, \rad^\omega_\sharp\mu_{j_2}) \geq \varepsilon_1 \}\right)
 \leq &\varepsilon_1^{-q}\int_{\S^{n-1}} 
 \mk_p^\R(\rad^\omega_\sharp\mu_{j_1}, \rad^\omega_\sharp\mu_{j_2})^qd\sigma_{n-1}(\omega)\\
 = &\varepsilon_1^{-q}\mk_{p, q}(\mu_{j_1}, \mu_{j_2})^q\\
 <&\varepsilon_2^q,
\end{align*}
for $j_1, j_2\geq J$. 
Now we can take a subsequence of $(\mu_j)_{j\in\mathbb{N}}$ (not relabeled) such that for all~$j\in\mathbb{N}$,
\begin{align*}
\sigma_{n-1}(\left\{\omega\in \S^{n-1}\mid \mk_p^\R(\rad^\omega_\sharp\mu_{j}, \rad^\omega_\sharp\mu_{j+1})\geq 2^{-j} \right\})\leq 2^{-j}.
\end{align*}
Setting 
\[
\Omega_{p, q}:=\S^{n-1}\setminus \left(\bigcap_{\ell=1}^\infty\bigcup_{j=\ell}^\infty \left\{\omega\in \S^{n-1}\Bigm| \mk_p^\R(\rad^\omega_\sharp\mu_{j}, \rad^\omega_\sharp\mu_{j+1})\geq 2^{-j} \right\}\right),
\]
we have
\begin{align*}
 \sigma_{n-1}(\Omega_{p, q})= 1-\sigma_{n-1}\left(\bigcap_{\ell=1}^\infty\bigcup_{j=\ell}^\infty \left\{\omega\in \S^{n-1}\Bigm| \mk_p^\R(\rad^\omega_\sharp\mu_{j}, \rad^\omega_\sharp\mu_{j+1})\geq 2^{-j} \right\}\right)=1
\end{align*}
by the Borel--Cantelli lemma, and we can see that the sequence $(\rad^\omega_\sharp\mu_j)_{j\in \N}$ is Cauchy in $\mk_p^\R$ whenever $\omega\in \Omega_{p, q}$. 
\hfill $\diamondsuit$

Since $\mk_p^\R$ is complete on $\mathcal{P}_p(\R)$ by Theorem~\ref{thm: wassconv}, 
for every $\omega\in \Omega_{p, q}$, 
there exists a measure $\mu^\omega\in \mathcal{P}_p(\R)$ such that 
\[
\lim_{j\to\infty}\mk_p^\R(\rad^\omega_\sharp\mu_j, \mu^\omega)=0,
\]
and 
$(\rad^\omega_\sharp\mu_j)_{j\in\mathbb{N}}$ weakly converges to $\mu^\omega$.

\noindent
{\bf Claim~2.} $(\mu_j)_{j\in \N}$ is a tight sequence.\\
{\it Proof of Claim~$2$.} Since $\sigma_{n-1}(\Omega_{p, q})=1$, the set $\Omega_{p, q}$ must contain some linearly independent collection $\{\omega_i\}_{i=1}^n$. 
For each~$1\leq i\leq n$, since $(\rad^{\omega_i}_\sharp\mu_j)_{j\in \N}$ weakly converges the sequence is tight, hence for any fixed $\varepsilon>0$, there exist $r_{i, \varepsilon}>0$ such that 
\[
\rad^{\omega_i}_\sharp\mu_j\left(\mathbb{R}\setminus[-r_{i, \varepsilon},r_{i, \varepsilon}]\right)<\varepsilon
\] 
for all $j\in \mathbb{N}$. 
By the independence of $\{\omega_i\}_{i=1}^n$, the set given by
\begin{align*}
 \bigcap_{i=1}^n\{x\in \R^n\mid \langle x, \omega_i\rangle \in [r_{i,\varepsilon}, r_{i,\varepsilon}] \},
\end{align*}
is compact and we see that
\begin{align*}
 \mu_j\left( 
 \R^n\setminus
 \bigcap_{i=1}^n\{ x\in \R^n\mid\langle x, \omega_i\rangle \in [r_{i,\varepsilon}, r_{i,\varepsilon}] \}\right)
 &=\mu_j\left(\bigcup_{i=1}^n\{x\in \R^n\mid \langle x, \omega_i\rangle \notin [-r_{i, \varepsilon},r_{i, \varepsilon}]\}\right)\\
 &\leq \sum_{i=1}^n\mu_j\left(\left\{x\in \R^n\mid \langle x, \omega_i\rangle\notin [-r_{i, \varepsilon},r_{i, \varepsilon}]\right\}\right)\\
 &=\sum_{i=1}^n \rad^{\omega_i}_\sharp\mu_j(
 \mathbb{R}\setminus[-r_{i, \varepsilon},r_{i, \varepsilon}])\\
 &<n\varepsilon,
\end{align*}
proving tightness.\hfill $\diamondsuit$

It follows from Claim~2 that there exists a probability measure $\mu\in \mathcal{P}(\R^n)$ and a subsequence (not relabeled) of $(\mu_j)_{j\in \N}$ that weakly converges to $\mu$.
Then $(\rad^\omega_\sharp\mu_j)_{j\in\mathbb{N}}$ converges weakly to $\rad^\omega_\sharp\mu$. 
By uniqueness of weak limits, we must have $\mu^\omega=\rad^\omega_\sharp\mu$ for all $\omega\in \Omega_{p, q}$, 
so in particular we have 
\begin{equation}\label{a}
\lim_{j\to\infty} \mk_p^\R(\rad^\omega_\sharp\mu_j, \rad^\omega_\sharp\mu)=0
\end{equation}
for $\omega \in \Omega_{p, q}$.

\noindent
{\bf Claim~3.} 
The $p$th moment of $(\mu_j)_{j\in \mathbb{N}}$ 
is uniformly bounded and
$\mu$ has finite $p$th moment.\\ 
{\it Proof of Claim~$3$.}
Let $\{\omega_i\}_{i=1}^n\subset \Omega_{p, q}$ be linearly independent. 
We denote by $(\omega_{ii'})_{1\leq i,i'\leq n}$  the Gram matrix of $\{\omega_i\}_{i=1}^n$, 
that is, 
$(\omega_{ii'})_{1\leq i,i'\leq n}=(\langle \omega_i, \omega_i' \rangle)_{1\leq i,i'\leq n}$.
Then $(\omega_{ii'})_{1\leq i,i'\leq n}$ is invertible and we denote by 
$(\omega^{ii'})_{1\leq i,i'\leq n}$ its inverse matrix.
%%%%
Since any norm on $\mathbb{R}^n$ is equivalent to the Euclidean norm, 
for $p$ and $n$, there exists $C_{n,p}>0$ such that 
\[
\sum_{i=1}^n a_i^2
\leq
\left(C_{n,p}\sum_{i=1}^n a_i^p\right)^{\frac{2}{p}}\quad
\text{for } a_i\in [0,\infty),
\ \text{ each } 1\leq i\leq n.
\]
For $x\in \mathbb{R}^n$, we have 
\[
x=\sum_{i=1}^n \sum_{i'=1}^n \omega^{ii'} \langle x,\omega_{i'}\rangle \omega_i 
\]
which yields 
\begin{align}\label{2top}
\begin{split}
|x|^2
&\leq n\sum_{i=1}^n \left( \sum_{i'=1}^n \lvert\omega^{ii'}\rvert\lvert \langle x,\omega_{i'}\rangle\rvert\right)^2\\
%%%%%
&\leq 
n^{3} \max_{1\leq i',i'' \leq n} |\omega^{i'i''}|^{2}
\sum_{i=1}^n |\langle x,\omega_{i}\rangle|^2 \\
&\leq
n^{3} \max_{1\leq i',i'' \leq n} |\omega^{i'i''}|^{2}
\left( C_{n,p}  \sum_{i=1}^n |\langle x,\omega_{i}\rangle|^p \right)^{\frac{2}{p}}.
\end{split}
\end{align}
%%%%
This ensures
\begin{align*}
\limsup_{j\to\infty}
\int_{\mathbb{R}^n} 
|x|^pd\mu_j(x)
&\leq 
n^{\frac{3p}{2}} \max_{1\leq i',i'' \leq n} |\omega^{i'i''}|^{p}
\limsup_{j\to\infty}
\int_{\mathbb{R}^n} 
C_{n,p} \sum_{i=1}^n \left| \langle x, \omega_{i} \rangle\right|^p
d\mu_j(x) \\
&=
n^{\frac{3p}{2}} C_{n,p} \max_{1\leq i',i'' \leq n} |\omega^{i'i''}|^{p}
\sum_{i=1}^n
\limsup_{j\to\infty}
\int_{\mathbb{R}} |t|^p d \rad_\sharp^{\omega_{i}}\mu_j(t)\\
&=n^{\frac{3p}{2}} C_{n,p} \max_{1\leq i',i'' \leq n} |\omega^{i'i''}|^{p}
\sum_{i=1}^n
\int_{\mathbb{R}} |t|^p d \rad_\sharp^{\omega_{i}}\mu(t)<\infty,
\end{align*}
where the last equality follows from Theorem~\ref{thm: wassconv}.
Then by Remark~\ref{rmk: mk compact} 
we see that $\mu$ has finite $p$th moment.
\hfill $\diamondsuit$

\noindent
{\bf Claim~4.} $\mk_{p, q}(\mu_j, \mu) \to0$ as $j\to\infty$.\\
{\it Proof of Claim~$4$.}
First suppose $q=\infty$ and fix $\varepsilon>0$. 
Then there exists some $J\in\mathbb{N}$ such that for $\sigma_{n-1}$-a.e. $\omega \in \mathbb{S}^{n-1}$,
all $j\geq J$, and any $\ell\in\mathbb{N}$,
\begin{align*}
\mk_p^{\R}(\rad^{\omega}_\sharp\mu_j,\rad^{\omega}_\sharp\mu_{j+\ell})
\leq 
\esssup_{\omega'\in\S^{n-1}}
\mk_p^{\R}(\rad^{\omega'}_\sharp\mu_j,\rad^{\omega'}_\sharp\mu_{j+\ell})
= \mk_{p,\infty}(\mu_j, \mu_{j+\ell})
 <\varepsilon.
\end{align*}
By \eqref{a},
we have
\[
\mk_p^{\R}(\rad^{\omega}_\sharp\mu_j,\rad^{\omega}_\sharp\mu)
=
\lim_{\ell\to \infty}\mk_p^{\R}(\rad^{\omega}_\sharp\mu_j,\rad^{\omega}_\sharp\mu_{j+\ell})
 <\varepsilon,
\]
then taking an essential supremum over $\omega$ shows
\[
 \mk_{p,\infty}(\mu_j, \mu)<\varepsilon,
\]
proving the claim.
If $q\neq\infty$, then for all $\omega\in \Omega_{p, q}$
\begin{align*}
\mk_p^\R( \rad^\omega_\sharp\mu_j, \rad^\omega_\sharp\mu)
&\leq \left( \int_{\R\times \R}
\left| t-s\right|^p 
d(\rad^\omega_\sharp\mu_j\otimes \rad^\omega_\sharp\mu)(t,s)\right)^{\frac{1}{p}}\\
&=\left(\int_{\R^n\times \R^n}\left| \langle x-y,\omega\rangle\right|^p d(\mu_j\otimes \mu)(x, y)\right)^{\frac{1}{p}}\\
&\leq 
2^{1-\frac{1}{p}}\left(\int_{\R^n}\left| x\right|^pd\mu_j(x)
+\int_{\R^n}\left|y\right|^p d\mu( y)\right)^{\frac{1}{p}},
\end{align*}
which we have shown is bounded independent of $j$. 
Thus if $q\neq\infty$, by the dominated convergence theorem together with~\eqref{a}, we obtain
\begin{align*}
\lim_{j\to\infty} \mk_{p, q}(\mu_j, \mu)
&=
\lim_{j\to\infty} \left(\int_{\mathbb{S}^{n-1}} \mk_p^\R(\rad^\omega_\sharp\mu_j, \rad^\omega_\sharp\mu)^q d\sigma_{n-1}(\omega)\right)^{\frac{1}{q}}= 0
 \end{align*}
as claimed.
\hfill $\diamondsuit$

Since the original sequence is Cauchy in $\mk_{p, q}$, convergence of a subsequence implies convergence of the full sequence to the same limit, and we are finished with the proof of completeness.
\end{proof}
%%%%
\subsection{Topological equivalence}
We now prove topological equivalence of $\mk_p^{\R^n}$ and $\mk_{p, q}$. In~ \cite{BayraktarGuo21}*{Theorem~2.3} it is shown that $\mk_1^{\R^n}$, $\mk_{1, 1}$, and $\mk_{1, \infty}$ are all topologically equivalent.
%%%%%%%%%%%
\begin{proof}[Proof of \nameref{thm: main sliced}~\eqref{thm: sliced top equiv}]
By Lemma~\ref{comparison}, any sequence that converges in $\mk_p^{\R^n}$ converges in $\mk_{p, q}$ to the same limit, for any $p$ and $q$.

Let $(\mu_j)_{j\in N}$ be a convergent sequence in $(\mathcal{P}_p(\R^n), \mk_{p,q})$
with limit $\mu$.
If $\Omega_{p, q}\subset \S^{n-1}$ is the set from Claim~1 in the proof of \nameref{thm: main sliced},  
there exists a linearly independent set $\{\omega_i\}_{i=1}^n\subset \Omega_{p, q}$ such that by Theorem~\ref{thm: wassconv},
\[
\lim_{r\to\infty}\limsup_{j\to\infty}
\int_{\mathbb{R}\setminus (-r, r)} |t|^p d \rad_\sharp^{\omega_i}\mu_j(t) \\
=0
\]
for $1\leq i\leq n$.
Set
\[
A_i:=\left\{x\in \mathbb{R}^n \mid \lvert\langle x,\omega_i\rangle\rvert
\geq \lvert\langle x,\omega_{i'}\rangle\rvert
\ \text{for }1\leq i'\leq n\right\}.
\]
For $x\in  A_{i_0} \setminus B_r^{\R^n}(0)$ with $1\leq i_0\leq n$ 
and $r>0$, 
we observe from \eqref{2top} that
\[
r^2<|x|^2\leq
n^3 \max_{1\leq i',i''\leq n} \lvert\omega^{i'i''}\rvert^2
\sum_{i=1}^n |\langle x,\omega_{i}\rangle|^2
\leq
n^4 \left( \max_{1\leq i',i''\leq n} \lvert\omega^{i'i''}\rvert^2\right)
|\langle x,\omega_{i_0}\rangle|^2.
\]
This leads to
\[
A_{i_0} \setminus B_r^{\R^n}(0)
\subset \left\{x\in \mathbb{R}^n \bigm|
 |\langle x,\omega_{i_0}\rangle|>\hat{c}^{-1}r
 \right\}
\]
for $r>0$,
where we set 
\[
\hat{c}=\hat{c}(\{\omega_{i}\}_{i=1}^n):=n^2\cdot  \max_{1\leq i', i'' \leq n} |\omega^{i'i''}|,
\]

then using~\eqref{2top},
\begin{align*}
\limsup_{j\to\infty}
\int_{\mathbb{R}^n\setminus B_r^{\R^n}(0)} 
|x|^pd\mu_j(x)
%%%%
&\leq 
n^{-\frac{p}{2}}\cdot C_{n,p}\hat{c}^{p}
 \cdot 
\limsup_{j\to\infty}
\int_{\mathbb{R}^n\setminus B_r^{\R^n}(0)} 
\sum_{i'=1}^n \left| \langle x,\omega_{i'} \rangle\right|^p
d\mu_j(x) \\
%%%%%
&\leq 
n^{-\frac{p}{2}}\cdot C_{n,p}\hat{c}^{p} \cdot 
\limsup_{j\to\infty}
\sum_{i=1}^n
\int_{A_i\setminus B_r^{\R^n}(0)} 
 \sum_{i'=1}^n \left| \langle x,\omega_{i'} \rangle\right|^p
d\mu_j(x) \\
%%%%
&
\leq 
n^{-\frac{p}{2}}\cdot C_{n,p}\hat{c}^{p} \cdot 
\limsup_{j\to\infty}
\sum_{i=1}^n
\int_{A_i\setminus B_r^{\R^n}(0)} 
n \left| \langle x,\omega_{i} \rangle\right|^p d\mu_j(x) \\
%%%
&
\leq 
n^{1-\frac{p}{2}}\cdot C_{n,p}\hat{c}^{p} \cdot 
\limsup_{j\to\infty}
\sum_{i=1}^n
\int_{\mathbb{R}\setminus (-\hat{c}^{-1}r, \hat{c}^{-1}r)} |t|^p d \rad_\sharp^{\omega_i}\mu_j(t) \\
&\xrightarrow{r\to \infty} 0.
\end{align*}
Thus by Theorem~\ref{thm: wassconv} we have the convergence 
of $(\mu_j)_{j\in N}$ in 
$(\mathcal{P}_p(\R^n), \mk_p^{\mathbb{R}^n})$ 
to~$\mu$.
\end{proof}
%%%%%%%%%%%%%%%
\subsection{Bi-Lipschitz non-equivalences}
In this subsection we will prove that the sliced and classical Monge--Kantorovich spaces (for the nontrivial case $n\geq 2$) are \emph{not} bi-Lipschitz equivalent for all $1\leq p<\infty$ when $q\neq \infty$, and for certain ranges of $p$ when $q=\infty$. In particular, this shows all sliced and max-sliced Wasserstein metrics are not bi-Lipschitz equivalent to the classical Monge--Kantorovich metrics.

We start by showing lower semi-continuity of $\mk_{p, q}$ under weak convergence. Although we do not explicitly use the following proposition, it is presented here as a result of independent interest.

\begin{proposition}\label{counterpart}
%%%%
For sequences $(\mu_j)_{j\in \mathbb{N}}$, $(\nu_j)_{j\in \mathbb{N}}$ in $\mathcal{P}_p(\mathbb{R}^n)$
weakly converging to $\mu_\infty$, $\nu_\infty$ respectively in $\mathcal{P}(\mathbb{R}^n)$,
\[
\mk_{p,q}(\mu_\infty, \nu_\infty)
\leq 
\liminf_{j\to \infty}\mk_{p,q}(\mu_j, \nu_j).
\]
Moreover, if 
\[
\sup_{j\in \N} \mk_{p,q}(\delta_0^{\mathbb{R}^n}, \mu_j)<\infty,
\]
then there exists a weakly convergent subsequence with limit in $\mathcal{P}_{\min\{p,q\}}(\mathbb{R}^n)$.
\end{proposition}
\begin{proof}
Let $(\mu_j)_{j\in \mathbb{N}}$, $(\nu_j)_{j\in \mathbb{N}}$ be sequences in $\mathcal{P}_p(\mathbb{R}^n)$
weakly converging to $\mu_\infty,\nu_\infty$ respectively in $\mathcal{P}(\mathbb{R}^n)$.
Then for any $\omega\in \S^{n-1}$, 
we see that $(\rad^\omega_\sharp\mu_j)_{j\in \mathbb{N}}$ and $(\rad^\omega_\sharp\nu_j)_{j\in \N}$ converge weakly to 
$\rad^\omega_\sharp\mu_\infty$ and $\rad^\omega_\sharp\nu_\infty$ respectively.
By weak lower-semicontinuity of $\mk_p^\R$, Remark~\ref{rmk: mk compact},  
we have 
\begin{align}\label{eq:counterpart}
\mk_p^\R(\rad^\omega_\sharp\mu_\infty, \rad^\omega_\sharp\nu_\infty)\leq \liminf_{j\to\infty} \mk_p^\R(\rad^\omega_\sharp\mu_j, \rad^\omega_\sharp\nu_j),
\end{align}
then it follows (from Fatou's lemma when $q\neq\infty$ and simple calculation when $q=\infty$) that 
\begin{align*}
\liminf_{j\to \infty}\mk_{p,q}(\mu_j, \nu_j)
&\geq \left\| \liminf_{j\to \infty}\mk_p^\R(\rad^\bullet_\sharp\mu_j, \rad^\bullet_\sharp\nu_j)\right\|_{L^q(\sigma_{n-1})}\\
&\geq \left\| \mk_p^\R(\rad^\bullet_\sharp\mu_\infty, \rad^\bullet_\sharp\nu_\infty)\right\|_{L^q(\sigma_{n-1})}
=\mk_{p,q}(\mu_\infty, \nu_\infty).
\end{align*}
This proves the first assertion.

Next, let $(\mu_j)_{j\in \mathbb{N}}$ be a sequence in $\mathcal{P}_p(\mathbb{R}^n)$ 
such that $\mk_{p,q}(\delta_0^{\mathbb{R}^n}, \mu_j)$ is uniformly bounded from above.
If  $p\leq q$, then by Lemma~\ref{comparison} with the monotonicity of $\mk_{p,q}$, 
\[
M_{p,n} \cdot \sup_{j\in \N} \mk_p^{\mathbb{R}^n}(\delta_0^{\mathbb{R}^n}, \mu_j)
=
 \sup_{j\in \N} \mk_{p,p}(\delta_0^{\mathbb{R}^n}, \mu_j)
\leq
\sup_{j\in \N} \mk_{p,q}(\delta_0^{\mathbb{R}^n}, \mu_j)
<\infty.
\]
Then from Remark~\ref{rmk: mk compact} there exists a weakly convergent subsequence with limit $\mu_\infty\in\mathcal{P}_p(\mathbb{R}^n)$.

On the other hand, if $q<p$, then 
\begin{align*}
\mk_{p,q}(\delta_0^{\mathbb{R}^n}, \mu_j)^q
&=\int_{\mathbb{S}^{n-1}}\left( \int_{\mathbb{R}^n} |\langle x, \omega\rangle|^p d\mu_j(x) \right)^{\frac{q}{p}}d\sigma_{n-1}(\omega)\\
&\geq
\int_{\mathbb{S}^{n-1}} \int_{\mathbb{R}^n} |\langle x, \omega\rangle|^q d\mu_j(x) d\sigma_{n-1}(\omega)\\
&=\mk_{q,q}(\delta_0^{\mathbb{R}^n}, \mu_j)^q
=M_{n,q}^q \cdot \mk_{q}^{\mathbb{R}^n}(\delta_0^{\mathbb{R}^n}, \mu_j)^q
\end{align*}
by Jensen's inequality and Tonelli's theorem, and then Lemma~\ref{comparison}. 
Again by Remark~\ref{rmk: mk compact},   
there exists a weakly convergent subsequence with limit $\mu_\infty\in\mathcal{P}_q(\mathbb{R}^n)$, completing the proof.
\end{proof}
By following an idea similar to \cite{BayraktarGuo21}*{Theorem 2.3 (iii)}, we can show the claimed metric bi-Lipschitz non-equivalences when $q\neq\infty$. For the case $q=\infty$, we rely on known generic estimates on $\mk_{1, \infty}$.

We first show estimates on the $p$-Monge--Kantorovich metrics between certain measures and empirical samples. Below, for $i\in \N$, we will write $\mathcal{H}^i$ for $i$-dimensional Hausdorff measure.
\begin{lemma}\label{lem: transform of surface measure}
For any $\omega=\cos\theta e_1+\sin\theta e_2\in \S^1$ with $\theta\in [0, {\pi}/{4}]$, 
the Borel probability measure $\rad^\omega_\sharp (4^{-1}\mathcal{H}^2|_{[-1,1]^2})$ is absolutely continuous with respect to the one-dimensional Lebesgue measure,
where the density $f_\theta$
is an even function on $\mathbb{R}$ and
\begin{align*}
    f_\theta(t):=
    \begin{dcases}
        \frac{1}{2\cos \theta},&0\leq t\leq \cos\theta-\sin\theta,\\
        \frac{\cos\theta+\sin\theta-t}{4\sin\theta\cos\theta}, &\cos\theta-\sin\theta< t\leq \cos\theta+\sin\theta,\\
        0,&t>\cos\theta+\sin\theta.
    \end{dcases}
\end{align*}
\end{lemma}
\begin{proof}
If $\theta=0$, it is clear that $f_\theta=2^{-1}\mathds{1}_{[-1, 1]}$, thus we may assume $\theta>0$. By symmetry we can see that $f_\theta$ will be even, and since $\theta\in (0, \pi/4]$ it is clear that
\[
\rad^\omega_\sharp (4^{-1}\mathcal{H}^2|_{[-1,1]^2})
([\cos\theta+\sin\theta, \infty))=0.
\]

Now fix $t\in (0, \cos\theta+\sin\theta)$. The density $f_\theta$ can be calculated as one quarter of the length of the line segment 
\begin{align*}
    \{(x, y)\in [-1, 1]^2\mid x\cos\theta+y\sin\theta=t\}.
\end{align*}
This line segment has one endpoint
\[
\left(\frac{t-\sin\theta}{\cos \theta}, 1\right)
\]
and the other endpoint $(x_\ast, y_\ast)$ with either $y_\ast=-1$ or $x_\ast=1$. 
This yields
\[
f_\theta(t)=\frac14\sqrt{ \left(x_\ast-\frac{t-\sin\theta}{\cos \theta}\right)^2 + (y_\ast-1)^2 }.
\]
If $y_\ast=-1$, then $x_\ast\cos \theta-\sin\theta=t$, or equivalently $x_\ast=(t+\sin\theta)/\cos \theta$.
However since $x_\ast\leq 1$, this can only happen when $t\leq \cos\theta-\sin\theta$. 
Thus when $\cos\theta-\sin\theta< t\leq \cos\theta+\sin\theta$, we have $x_\ast=1$, and we calculate $y_\ast=(t-\cos\theta)/\sin\theta$,
consequently
\begin{align*}
     f_\theta(t)
     &=\begin{dcases}
         \frac{1}{2\cos\theta},&0<t\leq \cos\theta-\sin\theta,\\
         \frac{\sin\theta+\cos\theta-t}{4\cos \theta\sin\theta},&\cos\theta-\sin\theta<t\leq \cos\theta+\sin\theta,
     \end{dcases}
 \end{align*}
 as claimed.
% %%%%%%%%%%%%%%%%%%%%%%%
\end{proof}
\begin{lemma}\label{lem: transform empirical sampling}
Fix $2\leq p<\infty$ and $n\geq 2$. 
Let us write $\pi_{\R^2}: \R^n\to \R^2\times\{0\}\subset \R^n$ for projection onto the first $2$ coordinates. 
Also let $\mu:=\left.4^{-1}\mathcal{H}^2\right\vert_{[-1, 1]^2\times \{0\}}$ viewed as an element of $\mathcal{P}_p(\R^n)$, 
and let $\{X_j\}_{j\in \N}$ be an i.i.d. collection of~$\R^n$-valued random variables %(defined on some probability space whose nature is immaterial), 
distributed according to $\mu$. Then for any $N\in \N$ and $\omega\in \S^{n-1}$ with $\pi_{\mathbb{R}^2}(\omega)\neq 0$,
\begin{align*}
    \mathbb{E}\left(\mk_p^\R\left(\rad^\omega_\sharp\mu, \frac1N\sum_{j=1}^N\delta_{\langle X_j,\omega\rangle}^\R\right)^p\right)
    &\leq (5p)^{p}\cdot 2^{p+1} \cdot N^{-\frac{p}{2}}.
    \end{align*}
%%%%%%%%%%%%%%%%%%%%
\end{lemma}
\begin{proof}
Fix $\omega\in \mathbb{S}^{n-1}$ with $\pi_{\mathbb{R}^2}(\omega)\neq 0$.
By symmetry, it is sufficient to consider the case 
\[
\theta:=\arctan\left( \frac{\langle e_2,\omega\rangle}{\langle e_1,\omega\rangle}\right)\in\left(0, \frac{\pi}{4}\right].
\]
We also write $\norm:=|\pi_{\mathbb{R}^2}(\omega)|$ and $\omega(\theta):=\cos \theta e_1 +\sin \theta e_2\in \mathbb{S}^1$, and define $L^\norm:\mathbb{R}\to \mathbb{R}$ by $L^\norm(t):=\norm t$ for $t\in \mathbb{R}$.
For any Borel measurable $\psi: \R\to\R$ by Lemma~\ref{lem: transform of surface measure} we have
     \begin{align*}
         \int_\R \psi d\rad^\omega_\sharp\mu
         &=\int_{\R^n}\psi(\langle x, \pi_{\R^2}(\omega)\rangle)d\mu(x)
         =\int_{\R^n}\psi\left(\norm \left\langle x, \norm^{-1}\pi_{\R^2}(\omega)\right\rangle\right)d\mu(x)\\
         &=\int_\R\psi (\norm t)f_{\theta}(t)dt
       =\int_\R\psi  dL^{\norm}_\sharp ( \rad^{\omega(\theta)}_\sharp (4^{-1}\mathcal{H}^2|_{[-1,1]^2})), 
     \end{align*}
consequently,
$\rad^\omega_\sharp\mu=L^{\norm}_\sharp ( \rad^{\omega(\theta)}_\sharp (4^{-1}\mathcal{H}^2|_{[-1,1]^2}))$.
Let $f_\omega$ denote the density of $R^\omega_\sharp\mu$ with respect to the one-dimensional Lebesgue measure, that is,
\[
f_\omega(t)=\norm^{-1} f_{\theta} (\norm^{-1}t),
\]
and set 
\[
F_\omega(t):=R^\omega_\sharp\mu((-\infty,t])=\rad^{\omega(\theta)}_\sharp (4^{-1}\mathcal{H}^2|_{[-1,1]^2})((-\infty,\norm^{-1} t])
\] 
for $t\in \mathbb{R}$.
Since every random variable $\langle X_j,\omega\rangle$ is distributed according to $\rad^\omega_\sharp\mu$, 
we observe from \cite{BobkovLedoux19}*{Theorem 5.3} that
\[
\mathbb{E}\left(\mk_p^\R\left(\rad^\omega_\sharp\mu, \frac1N\sum_{j=1}^N\delta_{\langle X_j,\omega\rangle}^\R\right)^p\right)
\leq
\left(\frac{5p}{\sqrt{N+2}}\right)^p
\cdot \int_{\mathbb{R}} \frac{[F_\omega(t) (1-F_\omega(t))]^{\frac{p}{2}}}{f_\omega(t)^{p-1}} dt.
\]
Since $f_{\theta}$ is an even function on $\mathbb{R}$,
we have $F_\omega(t)=1-F_\omega(-t)$ for $t\in \mathbb{R}$, hence
\begin{align*}
 \int_{\mathbb{R}} \frac{[F_\omega(s) (1-F_\omega(s))]^{\frac{p}{2}}}{f_\omega(s)^{p-1}} ds
&= 
2 \int_{0}^{\infty}\frac{[F_\omega(s) (1-F_\omega(s))]^{\frac{p}{2}}}{f_\omega(s)^{p-1}} ds
= 
2 \int_{0}^{\norm(\cos\theta+\sin\theta)} \frac{[F_\omega(s) (1-F_\omega(s))]^{\frac{p}{2}}}{f_\omega(s)^{p-1}} dt.
%%%%%
\end{align*}
For $ \cos\theta-\sin\theta<\norm^{-1}t\leq \cos\theta+\sin\theta$, by Lemma~\ref{lem: transform of surface measure} we have
\begin{align*}
1-F_\omega(t)=1-\rad^{\omega(\theta)}_\sharp (4^{-1}\mathcal{H}^2|_{[-1,1]^2})((-\infty,\norm^{-1} t])
=\int_{\norm^{-1} t}^\infty f_{\theta}(s)ds
=\frac12\cdot\frac{(\sin\theta+\cos \theta-\norm^{-1}t)^2}{4\cos\theta\sin\theta}.
\end{align*}
This implies that, again using Lemma~\ref{lem: transform of surface measure},
\begin{align*}
 \int_{\mathbb{R}} \frac{[F_\omega(s) (1-F_\omega(s))]^{\frac{p}{2}}}{f_\omega(s)^{p-1}} ds
&\leq
2 \int_{0}^{\norm(\cos\theta-\sin\theta)}  (2\norm\cos \theta)^{p-1}ds\\
&\quad+
2 \int_{\norm(\cos\theta-\sin\theta)}^{\norm(\cos\theta+\sin\theta)} 
2^{-\frac{p}{2}}\norm^{p-1}\left(4\cos\theta\sin\theta \right)^{\frac{p}{2}-1}\left(\sin \theta+\cos \theta -\norm^{-1}t\right) ds\\
%%%%
&=
2 \norm(\cos\theta-\sin\theta)  (2\norm\cos \theta)^{p-1}
%%%
+
 2^{2-\frac{p}{2}}\norm^{p}\left(4\cos\theta \right)^{\frac{p}{2}-1}\left(\sin \theta\right)^{\frac{p}{2}+1}\\
%%%%
&\leq
\norm^p(2^p +2^{\frac{p}{2}}\max\{1, 2^{\frac{1}{2}-\frac{p}{4}}\})
\leq (5p)^{p}\cdot 2^{p+1} \cdot N^{-\frac{p}{2}},
\end{align*}
finishing the proof.
\end{proof}
\begin{proof}[Proof of \nameref{thm: main sliced}~\eqref{thm: sliced not metric equiv}]
Let $n\neq 1$ and first assume $q\neq\infty$. 
Let $\mu:=\left.4^{-1}\mathcal{H}^2\right\vert_{[-1, 1]^2\times \{0\}}$
and again suppose $\{X_j\}_{j=1}^N$ are i.i.d. random samples distributed according to $\mu$; also write 
\[
\mu_N:=\frac1N\sum_{j=1}^N\delta_{X_j}^{\R^n}.
\] 
By using Lemma~\ref{lem: transform empirical sampling}, we have when $p\geq q$ by Minkowski's integral inequality,
\begin{align}
\begin{split}
    \mathbb{E}(\mk_{p, q}(\mu, \mu_N)^p)
    &=\mathbb{E}\left(\left[\int_{\S^{n-1}}\mk_p^\R(\rad^\omega_\sharp\mu, \rad^\omega_\sharp\mu_N)^qd\sigma_{n-1}(\omega)\right]^{\frac{p}{q}}\right)\\
    &\leq \left[\int_{\S^{n-1}}\left(\mathbb{E}\left(\mk_p^\R(\rad^\omega_\sharp\mu, \rad^\omega_\sharp\mu_N)^p\right)\right)^{\frac{q}{p}}d\sigma_{n-1}(\omega)\right]^{\frac{p}{q}}\\
    &\leq (5p)^{p}\cdot 2^{p+1}\cdot N^{-\frac{p}{2}}
    ,\label{eqn: E comparison q<p}
\end{split}    
\end{align}
while if $p<q$, using Jensen's inequality followed by monotonicity of $\mk_p^\R$ in $p$, then Tonelli's theorem, we have
\begin{align}
\begin{split}
    \mathbb{E}(\mk_{p, q}(\mu, \mu_N)^p)
=&    \mathbb{E}\left(\left[\int_{\S^{n-1}}\mk_p^\R(\rad^\omega_\sharp\mu, \rad^\omega_\sharp\mu_N)^qd\sigma_{n-1}(\omega)\right]^{\frac{p}{q}}\right)\\
&\leq \left(\mathbb{E}\left[\int_{\S^{n-1}}\mk_p^\R(\rad^\omega_\sharp\mu, \rad^\omega_\sharp\mu_N)^qd\sigma_{n-1}(\omega)\right]\right)^{\frac{p}{q}}\\
    &\leq \left(\mathbb{E}\left[\int_{\S^{n-1}}\mk_q^\R(\rad^\omega_\sharp\mu, \rad^\omega_\sharp\mu_N)^qd\sigma_{n-1}(\omega)\right]\right)^{\frac{p}{q}}\\
    &=\left[\int_{\S^{n-1}}\mathbb{E}\left(\mk_q^\R(\rad^\omega_\sharp\mu, \rad^\omega_\sharp\mu_N)^q\right)d\sigma_{n-1}(\omega)\right]^{\frac{p}{q}}\\
    &\leq  %[ (5q)^{q}\cdot(2^q +2)]^{\frac{p}{q}} \cdot N^{-\frac{p}{2}}
    [(5q)^{q}\cdot 2^{q+1}]^{\frac{p}{q}}\cdot N^{-\frac{p}{2}}
    .\label{eqn: E comparison p<q}
\end{split}
\end{align}
Now suppose by contradiction, there is a $C>0$ such that $\mk_p^{\R^n}(\mu, \nu)\leq C\mk_{p, q}(\mu, \nu)$ for all $\mu$, $\nu\in \mathcal{P}_p(\R^n)$, then by the classical Ajtai--Koml{\'o}s--Tusn{\'a}dy theorem (see for example~\cite{Ledoux17}*{(7)})
 we have for some constant $\widetilde{C}>0$ and for all $N$,
\begin{align*}
    \mathbb{E}(\mk_p^{\R^n}(\mu, \mu_N)^p)\geq \widetilde{C}\left(\frac{\log N}{N}\right)^{\frac{p}{2}}.
\end{align*}
Thus combining the above with~\eqref{eqn: E comparison q<p} or \eqref{eqn: E comparison p<q} implies for some constant $C_{p, q, n}>0$ depending only on $p$, $q$, and $n$,
\begin{align*}
    C_{p, q, n}\geq (\log N)^{\frac{p}{2}}
\end{align*}
which is a contradiction as $N\to\infty$, and the result is proved when $q\neq \infty$.

Now suppose $q=\infty$, let $\mu$ be the Lebesgue measure on $[0, 1]^n\subset \R^n$, and let $\mu_N$ be constructed from i.i.d. random samples distributed according to $\mu$ as above. By~\cite{Ledoux17}*{(7)}, we have
\begin{align*}
    \mathbb{E}(\mk^{\R^n}_p(\mu, \mu_N)^p)
    \geq \begin{cases}
    \widetilde{C} N^{-\frac{p}{2}}(\log N)^{\frac{p}{2}},&n=2,\\
        \widetilde{C} N^{-\frac{p}{n}},&n\geq 3.
    \end{cases}
\end{align*}
Also by~\cite{xu2022central}*{Theorem 3 (3)} and using that $\spt(\rad^\omega_\sharp\mu)$, $\spt(\rad^\omega_\sharp\mu_N)$ are all contained in $[-n^{1/2}, n^{1/2}]$, we have
\begin{align*}
    \mathbb{E}(\mk_{p, \infty}(\mu, \mu_N)^p)
    \leq (2n^{1/2})^{p-1}\mathbb{E}(\mk_{1, \infty}(\mu, \mu_N))\leq C'N^{-\frac{1}{2}},
\end{align*}
for some $C'>0$, thus combining with the above estimate and taking $N\to\infty$ would yield a contradiction if $\mk_p^{\R^n}(\mu, \mu_N)\leq C\mk_{p, \infty}(\mu, \mu_N)$ for some fixed constant $C>0$, and either $1\leq p<n/2$, or $n=2$ and $p=1$.
\end{proof}

\subsection{Non-existence of geodesics}
In this subsection we discuss the geodesic properties of $(\mathcal{P}_p(\R^n), \mk_{p, q})$. 
We now recall and/or prove a few basic properties of geodesics with respect to  $\mk_p^{\R^n}$. 

When $n=1$ and $p>1$, minimal geodesics are uniquely determined. 
\begin{lemma}[\cite{Santambrogio15}*{Theorem 2.9} and  \cite{Villani09}*{Corollary 7.23}]\label{1unique}
Let $p>1$.
For $\mu_0,\mu_1\in \mathcal{P}_p(\mathbb{R})$,
there exists a unique minimal geodesic in $(\mathcal{P}_p(\mathbb{R}),\mk_p^{\mathbb{R}})$
from $\mu_0$ to $\mu_1$.
\end{lemma}
We now take a short detour to highlight the  behavior of $\mk_{p, q}$ when $p=1$, which is different with respect to geodesics. First, it is easy to see that linear combinations give $\mk_1^{\R^n}$-minimal geodesics, unlike when $p>1$.
\begin{lemma}\label{geod1}
% Let $(X,\dist_X)$ be a complete, separable metric space.
For $\mu_0$, $\mu_1 \in \mathcal{P}_1(\R^n)$ 
the curve
$((1-\tau) \mu_0+\tau\mu_1)_{\tau\in [0,1]}$
is a $\mk_1^{\R^n}$-minimal geodesic.
\end{lemma}
\begin{proof}
Let $\gamma$ be a $1$-optimal coupling between $\mu_0$ and $\mu_1$.
For $\tau$, $\tau_1$, $\tau_2\in [0, 1]$ with $\tau_1\leq \tau_2$, set  
\[
\mu(\tau):=(1-\tau) \mu_0+\tau\mu_1, \quad
\gamma_{\tau_1,\tau_2}:=(\mathrm{Id}_{\R^n}\times\mathrm{Id}_{\R^n} )_\sharp ((1-\tau_2)\mu_0+ \tau_1 \mu_1)+ (\tau_2-\tau_1)\gamma.
\]
Then $\mu:[0,1]\to \mathcal{P}_1({\R^n})$ 
is a curve joining $\mu_0$ to $\mu_1$, and $\gamma_{\tau_1,\tau_2}\in \mathcal{P}({\R^n} \times {\R^n})$.
For any Borel set $A\subset {\R^n}$, we can see 
\[
\gamma_{\tau_1,\tau_2}(A\times {\R^n})=(1-\tau_2)\mu_0(A)+\tau_1 \mu_1(A)+(\tau_2-\tau_1)\mu_0(A)=\mu(\tau_1)(A)
\]
and similarly
$\gamma_{\tau_1,\tau_2}({\R^n}\times A)=\mu(\tau_2)(A)$.
Thus $\gamma_{\tau_1,\tau_2}\in \Pi(\mu(\tau_1),\mu(\tau_2))$ and 
\begin{align*}
\mk_1^{\R^n}(\mu(\tau_1),\mu(\tau_2))
&\leq 
\int_{{\R^n} \times {\R^n}} \dist_{\R^n}(t,s) d\gamma_{\tau_1,\tau_2}(t,s)\\
&=(\tau_2-\tau_1)
\int_{{\R^n} \times {\R^n} } \dist_{\R^n}(t, s) d\gamma(t,s)
=(\tau_2-\tau_1)\mk_1^{\R^n}(\mu_0,\mu_1),
\end{align*}
proving the lemma.
\end{proof}

With this lemma in hand, we can easily prove that $(\mathcal{P}_1(\mathbb{R}^n), \mk_{1,q})$ is geodesic for all~$q$.
\begin{proposition}\label{geod}
For any $\mu_0$, $\mu_1\in \mathcal{P}_1(\R^n)$, the curve $((1-\tau)\mu_0+\tau\mu_1)_{\tau\in [0, 1]}$ is a $\mk_{1, q}$-minimal geodesic.
\end{proposition}
\begin{proof}
For $\mu_0$, $\mu_1 \in \mathcal{P}_1(\mathbb{R}^n)$, and $\tau\in [0,1]$, 
set $\mu(\tau):=(1-\tau) \mu_0 +\tau \mu_1$.
Then, for $\omega \in \mathbb{S}^{n-1}$,
\[
 (\rad^\omega_\sharp \mu(\tau))_{\tau\in [0,1]}=((1-\tau)\rad^\omega_\sharp \mu_0+\tau\rad^\omega_\sharp \mu_1)_{\tau\in [0,1]}
\]
is a $\mk_1^{\mathbb{R}}$-minimal geodesic by Lemma~\ref{geod1}.
Then, for $\tau_1$, $\tau_2\in [0,1]$, 
we have
\begin{align*}
\mk_{1,q} (\mu(\tau_1),\mu(\tau_2))
=
\left\| \mk_1^\R ( \rad^\omega_\sharp \mu(\tau_1), \rad^\omega_\sharp \mu(\tau_2) ) \right\|_{L^q(\sigma_{n-1})}
%%%
=|\tau_1-\tau_2|
\mk_{1,q} (\mu_0,\mu_1).
\end{align*}
\end{proof}
\begin{remark}
    As mentioned in the introduction, although the above shows $(\mathcal{P}_1(\R^n), \mk_{1, q})$ is a geodesic space for any $1\leq q\leq \infty$, the minimal geodesics are merely \emph{linear} convex combinations, hence do not reflect any particular transport properties. 
\end{remark}
The situation when $1<p<\infty$ paints a stark contrast with the case $p=1$.
\begin{proof}[Proof of \nameref{thm: main sliced}~\eqref{thm: sliced not geodesic}]
Let $\mu_0$, $\mu_1\in \mathcal{P}_p(\mathbb{R}^n)$.
Assume that there exists
a minimal geodesic $\mu:[0,1]\to (\mathcal{P}_p(\mathbb{R}^n),\mk_{p,q}) $ 
from $\mu_0$ to $\mu_1$.
Then, for any partition $\{[\tau_i, \tau_{i+1})\}_{i=0}^{N-1}$ of $[0,1)$, we have 
\begin{align*}
\mk_{p,q} (\mu_0,\mu_1)
&=
\sum_{i=1}^N \mk_{p,q} (\mu(\tau_{i-1}),\mu(\tau_i))\\
&=
\sum_{i=1}^N \left\| \mk_p^{\mathbb{R}}( \rad^\bullet_\sharp \mu(\tau_{i-1}), \rad^\bullet_\sharp \mu(\tau_i)) \right\|_{L^q(\sigma_{n-1})}\\
&
\geq 
\left\| \sum_{i=1}^N \mk_p^{\mathbb{R}}( \rad^\bullet_\sharp \mu(\tau_{i-1}), \rad^\bullet_\sharp \mu(\tau_i)) \right\|_{L^q(\sigma_{n-1})}\\
&
\geq 
\left\| \mk_p^{\mathbb{R}}( \rad^\bullet_\sharp \mu_0, \rad^\bullet_\sharp \mu_1 ) \right\|_{L^q(\sigma_{n-1})}\\
%%%
&=\mk_{p,q} (\mu_0,\mu_1),
\end{align*}
consequently the inequalities above become equalities.
For $q\neq\infty$, this implies that
\[
\sum_{i=1}^N \mk_p^{\mathbb{R}}( \rad^\omega_\sharp \mu(\tau_{i-1}), \rad^\omega_\sharp \mu(\tau_i)) = \mk_p^{\mathbb{R}}( \rad^\omega_\sharp \mu_0, \rad^\omega_\sharp \mu_1 )
\]
for $\sigma_{n-1}$-a.e.\ $\omega$, 
and then for all $\omega\in \mathbb{S}^{n-1}$
by the continuity in Lemma~\ref{prelemma}~\eqref{conti}. In particular, $\rad^{\omega}_\sharp \mu(\cdot):[0,1]\to (\mathcal{P}_p(\mathbb{R}), \mk_p^{\mathbb{R}})$ is the (unique by Lemma~\ref{1unique}) minimal geodesic
from $R^{\omega}_\sharp \mu_0$ to~$R^{\omega}_\sharp \mu_1$ for every $\omega\in \S^{n-1}$.

Regarding the case $q=\infty$, notice that 
the continuity in Lemma~\ref{prelemma}~\eqref{conti} together with the compactness of $\mathbb{S}^{n-1}$
ensures the existence of $\omega_\ast \in \mathbb{S}^{n-1}$ such that 
\begin{align*}
\mk_{p,\infty} (\mu_0,\mu_1)
=\mk_p^{\mathbb{R}} ( \rad^{\omega_\ast}_\sharp \mu_0, \rad^{\omega_\ast}_\sharp \mu_1 ).
\end{align*}
Then it holds for such $\omega_\ast$ that 
\begin{align*}
\mk_{p,\infty} (\mu_0,\mu_1)
=\mk_p^{\mathbb{R}} ( \rad^{\omega_\ast}_\sharp \mu_0, \rad^{\omega_\ast}_\sharp \mu_1 )
&\leq 
\sum_{i=1}^N \mk_p^{\mathbb{R}}( \rad^{\omega_\ast}_\sharp \mu(\tau_{i-1}), \rad^{\omega_\ast}_\sharp \mu(\tau_i)) \\
&\leq
\left\| \sum_{i=1}^N \mk_p^{\mathbb{R}}( \rad^\bullet_\sharp \mu(\tau_{i-1}), \rad^\bullet_\sharp \mu(\tau_i)) \right\|_{L^\infty(\sigma_{n-1})}\\
&=\mk_{p,\infty} (\mu_0,\mu_1).
\end{align*}
Thus from Lemma~\ref{1unique}, 
for $1< p < \infty$, 
if $\mk_{p,\infty} (\mu_0,\mu_1)=\mk_p^{\mathbb{R}} ( \rad^{\omega_\ast}_\sharp \mu_0, \rad^{\omega_\ast}_\sharp \mu_1 )$, for some $\omega_\ast\in \S^{n-1}$, 
then
$\rad^{\omega_\ast}_\sharp \mu(\cdot):[0,1]\to (\mathcal{P}_p(\mathbb{R}), \mk_p^{\mathbb{R}})$ is the unique minimal geodesic from $R^{\omega_\ast}_\sharp \mu_0$ to~$R^{\omega_\ast}_\sharp \mu_1$.

Now let us take 
\[
\rast:=2-\sqrt{2}\in (0,1)
\]
and 
\[
\mu_0:=\frac14\left(  \delta^{\mathbb{R}^n}_{e_1+e_2}+  \delta^{\mathbb{R}^n}_{-e_1+e_2}+  \delta^{\mathbb{R}^n}_{-e_1-e_2}+  \delta^{\mathbb{R}^n}_{e_1-e_2}   \right),\qquad
\mu_1:=\frac14\left(  \delta^{\mathbb{R}^n}_{\rast e_1}+  \delta^{\mathbb{R}^n}_{\rast e_2}+  \delta^{\mathbb{R}^n}_{-\rast e_1}+  \delta^{\mathbb{R}^n}_{-\rast e_2}   \right).
\]
We first consider the case $q\neq\infty$. In this case, for  $\omega=e_1$, $e_2$,
since $(\rad^{\omega}_\sharp \mu(\tau))_{\tau\in[0,1]}$
is a minimal geodesic in $(\mathcal{P}_p(\mathbb{R}), \mk_p^{\mathbb{R}})$ 
and 
\[
\rad^{\omega}_\sharp \mu(0)=
\rad^{\omega}_\sharp \mu_0=\frac12\left( \delta^{\mathbb{R}}_{1}+\delta^{\mathbb{R}}_{-1}\right),\quad
\rad^{\omega}_\sharp \mu(1)=
\rad^{\omega}_\sharp \mu_1=
\frac14\left(  \delta^{\mathbb{R}}_{\rast}+  2\delta^{\mathbb{R}}_{0}+\delta^{\mathbb{R}}_{-\rast}\right),
\]
we have
\[
\rad^{\omega}_\sharp \mu(1/2)=
\frac14\left(  \delta^{\mathbb{R}}_{\frac{1+\rast}{2}}+  \delta^{\mathbb{R}}_{\frac{1}{2}}+\delta^{\mathbb{R}}_{-\frac{1}{2}}+\delta^{\mathbb{R}}_{-\frac{1+\rast}{2}}\right).
\]
This implies that 
\begin{align*}
1
%%%%%%
&=
\rad^\omega_\sharp \mu (1/2)
\left(\left\{ \pm\frac{1+\rast}{2}, \pm \frac{1}{2}\right\}\right)
 %%%%
=\mu (1/2) 
\left(
\left\{ x\in \mathbb{R}^n \bigm|\ \langle x, \omega \rangle
\in
\left\{ \pm\frac{1+\rast}{2}, \pm\frac{1}{2}\right\}
\right\}\right),
\end{align*}
%%%%
consequently  the support of $\mu(1/2)$ is contained in
\begin{align*}
E:=&\bigcap_{i=1}^2
\left\{  x\in \mathbb{R}^n \bigm| \langle e_i,x \rangle
\in
\left\{ \pm\frac{1+\rast}{2}, \pm \frac{1}{2}\right\}
\right\}.
\end{align*}
Similarly,  for  $\omega=(e_1+ e_2)/\sqrt{2}$,
we find that 
\[
\rad^{\omega}_\sharp \mu(1/2)=
\frac14\left(  \delta^{\mathbb{R}}_{\frac{2+\rast}{2\sqrt{2}}}+  \delta^{\mathbb{R}}_{\frac{\rast}{2\sqrt{2}}}
+ \delta^{\mathbb{R}}_{-\frac{\rast}{2\sqrt{2}}}+\delta^{\mathbb{R}}_{-\frac{2+\rast}{2\sqrt{2}}}\right)
\]
and  the support of $\mu(1/2)$ is contained in
\begin{align*}
E':=
\left\{  x\in \mathbb{R}^n \bigm| \langle e_1+e_2,x \rangle
\in
\left\{ \pm \frac{2+\rast}{2}, \pm \frac{\rast}{2}\right\}
\right\}.
\end{align*}
However this implies that 
the support of $\mu(1/2)$ is contained in $E\cap E'=\emptyset$, which is clearly a contradiction.
Thus there is no minimal geodesic in $(\mathcal{P}_p(\mathbb{R}^n),\mk_{p,q})$ from $\mu_0$ to $\mu_1$.

To handle the case $q=\infty$, we only need to verify that $e_1$, $e_2$, and $2^{-1/2}(e_1+e_2)$ are maximizers of $\omega\mapsto \mk_p^\R(\rad^\omega_\sharp\mu_0, \rad^\omega_\sharp\mu_1)$, and then we will reach the same contradiction as above. To this end we see that 
for $\omega\in \mathbb{S}^{n-1}$, there exist $\theta \in [0,2\pi)$ and $\norm\in [0,1]$ (possibly not unique) such that 
\[
\langle e_1,\omega\rangle =\norm \langle  e_1,\omega(\theta)\rangle, 
\quad
\langle e_2,\omega\rangle =\norm \langle  e_1,\omega(\theta)\rangle, 
\]
where we set 
\[
\omega(\theta):=\cos \theta e_1+\sin \theta e_2.
\]
This implies that
\begin{align*}
\mk_p^\mathbb{R}(R_\sharp^{\omega}\mu_0, R_\sharp^{\omega}\mu_1)
= \norm\mk_p^\mathbb{R}\left( R_\sharp^{\omega(\theta)}\mu_0 ,  R_\sharp^{\omega(\theta)}\mu_1\right)
\leq \mk_p^\mathbb{R}\left( R_\sharp^{\omega(\theta)}\mu_0 ,  R_\sharp^{\omega(\theta)}\mu_1\right),
\end{align*}
thus we conclude that 
\[
\mk_{p,\infty}(\mu_0,\mu_1)
=\sup_{\theta \in [0, 2\pi)} \mk_p^\mathbb{R}(R_\sharp^{\omega(\theta)}\mu_0, R_\sharp^{\omega(\theta)}\mu_1).
\]
For $\theta \in [0, 2\pi)$, 
from 
\begin{align*}
R_\sharp^{\omega(\theta)}\mu_0
&=\frac14\left( \delta^{\mathbb{R}}_{\cos \theta+\sin \theta}+  \delta^{\mathbb{R}}_{-\cos \theta+\sin \theta}+  \delta^{\mathbb{R}}_{-\cos \theta-\sin \theta}+  \delta^{\mathbb{R}}_{\cos \theta-\sin \theta}   \right),\\
%%%
R_\sharp^{\omega(\theta)}\mu_1
&=\frac14\left(  \delta^{\mathbb{R}}_{\rast\cos \theta}+  \delta^{\mathbb{R}}_{\rast\sin \theta}+  \delta^{\mathbb{R}}_{-\rast\cos \theta}+  \delta^{\mathbb{R}}_{-\rast\sin \theta}   \right),
%%%
\end{align*}
we observe that 
\begin{align*}
\mk_p^\mathbb{R}(R_\sharp^{\omega(\theta)}\mu_0, R_\sharp^{\omega(\theta)}\mu_1)
=
\mk_p^\mathbb{R}(R_\sharp^{\omega(\theta+\frac{\pi}{2})}\mu_0, 
R_\sharp^{\omega(\theta+\frac{\pi}{2})}\mu_1
) 
=
\mk_p^\mathbb{R}(
R_\sharp^{\omega(-\theta)}\mu_0, 
R_\sharp^{\omega(-\theta)}\mu_1
),
\end{align*}
which implies that 
\[
\mk_{p,\infty}(\mu_0,\mu_1)
=\max_{\theta \in [0,\pi/4]} \mk_p^\mathbb{R}(R_\sharp^{\omega(\theta)}\mu_0, R_\sharp^{\omega(\theta)}\mu_1).
\]

For $\theta \in [0,\pi/4]$, define 
\begin{align*}
w_p(\theta)&:=\mk_p^\mathbb{R}(R_\sharp^{\omega(\theta)}\mu_0, R_\sharp^{\omega(\theta)}\mu_1)^p.
\end{align*}
Then we calculate 
\begin{align*}
w_p(0)
&=\mk_p^\mathbb{R}\left(\frac12\left( \delta^{\mathbb{R}}_{1}+\delta^{\mathbb{R}}_{-1}\right),
\frac14\left(  \delta^{\mathbb{R}}_{\rast}+  2\delta^{\mathbb{R}}_{0}+\delta^{\mathbb{R}}_{-\rast}\right)\right)^p\\
&=\frac12 (1-\rast)^p+\frac12,\\
w_p\left(\frac{\pi}{4}\right)
%%%
&=\mk_p^\mathbb{R}\left(
\frac14\left( \delta^{\mathbb{R}}_{\sqrt{2}}+  2\delta^{\mathbb{R}}_{0}+  \delta^{\mathbb{R}}_{-\sqrt{2}}  \right),
\frac12\left(\delta^{\mathbb{R}}_{\frac{\rast}{\sqrt{2}}}+  \delta^{\mathbb{R}}_{-\frac{\rast}{\sqrt{2}}} \right) \right)^{p}\\
&=\frac12 \left(\sqrt{2}-\frac{\rast}{\sqrt{2}}\right)^p+\frac12\left(\frac{\rast}{\sqrt{2}}\right)^p.
\end{align*}
Since 
\[
1-\rast=\sqrt{2}-1=\frac{\rast}{\sqrt{2}}, \quad
\sqrt{2}-\frac{\rast}{\sqrt{2}}=\sqrt{2}-\frac{2-\sqrt{2}}{\sqrt{2}}=1,
\]
we conclude 
\[
w_p(0)=w_p(\pi/4).
\]
Note that,  
if we replace $\rast$ with any $r \in (0,1)$, then $w_p(0)=w_p(\pi/4)$ holds if and only if $r=\rast$ 
in the case of $p\neq 2$, but $w_p(0)=w_p(\pi/4)$ holds for any $r\in(0,1)$ in the case of $p=2$. 

Next we will prove  $w_p(\theta)< w_p(0)$ for $\theta\in (0,\pi/4)$.
Indeed, for $\theta \in (0,\pi/4)$, we have
\begin{gather*}
-\cos \theta-\sin \theta \leq   -\cos \theta+\sin \theta \leq  \cos \theta-\sin \theta \leq \cos \theta+\sin \theta , \\
-\rast\cos \theta \leq - \rast\sin \theta \leq \rast\sin \theta \leq \rast\cos \theta,
\end{gather*}
which yields
\begin{align*}
w_p(\theta)
&=
\frac12\left\{
\left(
\cos \theta+\sin \theta-\rast\cos \theta 
\right)^p
+\left|\cos \theta-\sin \theta-\rast\sin \theta\right|^p
\right\}\\
&=
\frac12\left[
\left\{
(1-\rast)\cos \theta+\sin \theta\right\}^p
+\left|\cos \theta-(1+\rast)\sin \theta\right|^p
\right],
\end{align*}
where we have used that $\theta \mapsto (1-\rast)\cos \theta+\sin \theta $ is strictly increasing on $(0,\pi/4)$.
Let 
\[
\theta_\ast:=\arctan\frac{1}{1+\rast}\in \left(0, \frac{\pi}{4}\right).
\]
On one hand, on $(0,\theta_\ast)$,
\[
\theta \mapsto \left|\cos \theta-(1+\rast)\sin \theta\right|=\cos \theta-(1+\rast)\sin \theta
\]
is strictly decreasing.
On the other hand,   on $(\theta_\ast,\pi/4)$,
\[
\theta \mapsto \left|\cos \theta-(1+\rast)\sin \theta\right|=-\cos \theta+(1+\rast)\sin \theta
\]
is strictly increasing. These imply that  $w_p(\theta) \leq w_p(\pi/4)=w_p(0)$ for $\theta \in (\theta_\ast, \pi/4)$.

Now assume $\theta \in (0, \theta_\ast)$ and set
\[
\alpha(\theta):=(1-\rast)\cos \theta+\sin \theta,\quad
\beta(\theta):=\cos \theta-(1+\rast)\sin \theta.
\]
Then 
$\alpha(\theta)$, $\beta(\theta)>0$ and  $w_p(\theta)=(\alpha(\theta)^p+\beta(\theta)^p)/2$.
We also find 
\begin{align*}
\alpha'(\theta)&=-(1-\rast)\sin \theta+\cos \theta=\frac{1}{\sqrt{2}}\left( \alpha(\theta)+\beta(\theta)\right),\\
\beta'(\theta)&=-\sin \theta-(1+\rast)\cos \theta=-\frac{1}{\sqrt{2}}\left( 3\alpha(\theta)+\beta(\theta)\right).
\end{align*}
Consequently, we have
\begin{align*}
w_p'(\theta)
&=
\left\{\frac12(\alpha(\theta)^p+\beta(\theta)^p)\right\}'
=\frac{p}{2\sqrt{2}}\left\{\alpha(\theta)^{p-1}\left( \alpha(\theta)+\beta(\theta)\right)-\beta(\theta)^{p-1}\left( 3\alpha(\theta)+\beta(\theta)\right)\right\}\\
&=\frac{p}{2\sqrt{2}}\beta(\theta)^p F_p(\alpha(\theta)\beta(\theta)^{-1}),
\end{align*}
where we define $F_p:(0,\infty)\to \mathbb{R}$ by
\[
F_p(\uvariable):=\uvariable^p+\uvariable^{p-1}-3\uvariable-1.
\]
Note that $\theta \mapsto\alpha(\theta)\beta(\theta)^{-1}$ is strictly increasing on $(0,\theta_\ast)$ and 
\[
 \alpha(0)\beta(0)^{-1}=1-\rast\qquad
\lim_{\theta \uparrow \theta_\ast}  \alpha(\theta)\beta(\theta)^{-1}=\infty.
\]
A direct calculation gives
\[
F'_p(\uvariable)=p\uvariable^{p-1}+(p-1)\uvariable^{p-2}-3, \quad
F''_p(\uvariable)=(p-1)\uvariable^{p-3}[p\uvariable+(p-2)],
\]
and since $p>1$ we see $F_p$ is strictly convex on $(1,\infty)$.
We also have 
\[
F_p(\uvariable)\leq \uvariable+1-3\uvariable-1=-2\uvariable<0\  \text{if } \uvariable\in (0,1], \quad
\lim_{\uvariable\to \infty}F_p(\uvariable)=\infty,
\]
thus there exists a unique $\uvariable_p\in (1,\infty)$ such that
$F_p(\uvariable)<0$ if $s\in (0,\uvariable_p)$ and $F_p(\uvariable)>0$ if $\uvariable\in (\uvariable_p,\infty)$.
Equivalently,
there exists a unique $\theta_p\in (0,\theta_\ast)$ such that
$w_p'(\theta)<0$ if $\theta\in (0,\theta_p)$ and $w_p'(\uvariable)>0$ if $\theta\in (\theta_p,\theta_\ast)$.
%%%
Consequently, we have $w_p(\theta) <\max\{ w_p(0)  ,w_p(\theta_\ast)\}=w_p(0)$ for $\theta \in (0,\theta_\ast)$. 
Thus we have
\[
w_p(\theta)< w_p(0)  \quad \text{for }  \theta\in (0,\pi/4).
\]
All together the above implies that for $\omega\in \mathbb{S}^{n-1}$, 
\[
\mk_{p,\infty}(\mu_0,\mu_1)=
\mk_p^\mathbb{R}(R_\sharp^{\omega}\mu_0, R_\sharp^{\omega}\mu_1)
\]
holds if and only if
\[
\omega 
\in
\left\{
\pm e_1, \pm e_2, 
\frac{1}{\sqrt{2}} (e_1 \pm e_2),
\frac{1}{\sqrt{2}} (-e_1 \pm e_2)
\right\},
\]
completing the desired verification, hence the proof.
\end{proof}
%%%%%%%%%%%%%%%%%%%%%%%%
\subsection{Existence of sliced barycenters}
Finally, we prove the existence of $\mk_{p, q}$-barycenters by a simple compactness argument.
\begin{proof}[Proof of \nameref{thm: main sliced}~\eqref{thm: barycenters}]
For ease of notation, we write
\begin{align*}
    B(\nu):=\sum_{k=1}^K\lambda_k \mk_{p, q}(\mu_k, \nu)^{\barypower}.
\end{align*}
If $\barypower=0$, we see $B$ is constant on $\mathcal{P}_p(\mathbb{R}^n)$
 and the claim holds trivially, thus assume $\barypower\neq 0$.
 Since each $\mu_k\in \mathcal{P}_p(\R^n)$ and $B$ is nonnegative, it has a finite infimum and we may take a minimizing sequence $(\nu_j)_{j\in \mathbb{N}}\subset \mathcal{P}_p(\R^n)$; that is 
\[
\lim_{j\to\infty}B(\nu_j)=\inf_{\nu \in \mathcal{P}_p(\R^n)}B(\nu),
\]
and moreover we may assume $\sup_j B(\nu_j)\leq C<\infty$ for some $C>0$. Since $\lambda_1>0$, then for any $j$, using Lemma \ref{comparison} and then H\"older's inequality and that $p\leq q$, we have
\begin{align*}
 \left(\int_{\R^n}\left| x\right|^pd\nu_j(x)\right)^{\frac{1}{p}}
 &=\mk_p^{\R^n}(\delta_0^{\R^n}, \nu_j)=M_{p, n}\mk_{p, p}(\delta_0^{\R^n}, \nu_j)\\
 &\leq M_{p, n}\mk_{p, q}(\delta_0^{\R^n}, \nu_j)\\
 &\leq M_{p, n}\left( \mk_{p, q}(\delta_0^{\R^n}, \mu_1)+\mk_{p, q}(\mu_1, \nu_j)\right)\\
 &\leq M_{p, n}\left(\mk_{p, q}(\delta_0^{\R^n}, \mu_1)+(\lambda_1^{-1}C)^{\frac{1}{\barypower}}\right).
\end{align*}
Since the $p$th moments of the $\nu_j$ are uniformly bounded, we may pass to a subsequence (not relabeled) which converges weakly to some $\nu_\infty\in \mathcal{P}_p(\R^n)$ as $j\to\infty$ 
by Remark~\ref{rmk: mk compact}. 
Then as in \eqref{eq:counterpart}, we have
\[
\mk_p^\R(\rad^\omega_\sharp\mu_k, \rad^\omega_\sharp\nu_\infty)\leq \liminf_{j\to\infty} \mk_p^\R(\rad^\omega_\sharp\mu_k, \rad^\omega_\sharp\nu_j).
\]
Thus by Fatou's lemma, we obtain
\begin{align*}
 B(\nu_\infty)&=\sum_{k=1}^K\lambda_k \left\|\mk_p^{\R}(\rad^\bullet_\sharp\mu_{k}, \rad^\bullet_\sharp\nu_\infty)\right\|_{L^q(\sigma_{n-1})}^{\barypower}\\
 &\leq \sum_{k=1}^K\lambda_k\liminf_{j\to\infty} \left\|\mk_p^{\R}(\rad^\bullet_\sharp\mu_{k}, \rad^\bullet_\sharp\nu_j)\right\|_{L^q(\sigma_{n-1})}^{\barypower}\\
 &\leq \liminf_{j\to\infty}\sum_{k=1}^K\lambda_k \left\|\mk_p^{\R}(\rad^\bullet_\sharp\mu_{k}, \rad^\bullet_\sharp\nu_j)\right\|_{L^q(\sigma_{n-1})}^{\barypower}=\liminf_{j\to\infty}B(\nu_j),
\end{align*}
proving that $\nu_\infty$ is an 
$\mk_{p, q}$-barycenter of $\{\mu_k\}_{k=1}^K$.
\end{proof}
%%%%%%%
\subsection{Duality}
We will now formulate a dual problem for $\mk_{p, q}$ for $p\leq q$ in several steps. First  recall the following definition in a general setting.
%%%
\begin{definition}\label{def: c-transform}
For a metric space $(X, \dist_X)$ and $\phi\in C_b(X)$, we set
\[
\|\phi\|_{C_b(X)}
:=\sup_{s\in X} |\phi(s)|.
\]
%%%%
Then for a function $\phi$ on $X$ and $s\in X$,
define the \emph{$\dist_X^p$-transform of $\phi$} by
\begin{align*}
\phi^{\dist_X^p}(s):=\sup_{t\in X} \left(-\dist_X(t,s)^p-\phi(t)\right) \in (-\infty,\infty].
\end{align*}
\end{definition}

We start with an elementary inequality, which we prove in the general setting of a metric space.
%%%%%%%%%
\begin{lemma}\label{cpp}
Let $(X, \dist_X)$ be a metric space.
If $\phi, \psi\in C_b(X)$, we have $\phi^{\dist_X^p}, \psi^{\dist_X^p}\in C_b(X)$ with
\[
 \| \phi^{\dist_X^p}\|_{C_b(X)} \leq \|\phi\|_{C_b(X)} 
\]
and $\|\phi^{\dist_X^p}-\psi^{\dist_X^p}\|_{C_b(X)} \leq \|\phi-\psi\|_{C_b(X)}$.
\end{lemma}
\begin{proof}
For any $s\in X$, we find that
\begin{align*}
\lVert \phi\rVert_{C_b(X)}\geq\sup_{t\in X}(-\phi(t))\geq \phi^{\dist_X^p}(s)
 \geq -\dist_X(s,s)^p-\phi(s)
 =-\phi(s)
 \geq 
 -\|\phi\|_{C_b(X)},
 \end{align*}
thus $\phi^{\dist_X^p}$ is bounded on $X$. 

Next suppose $(s_j)_{j\in \mathbb{N}}$ is a sequence in $X$ converging to some $s_0$.
Fix $\varepsilon>0$.
Then since $\phi^{\dist_X^p}$ is bounded from above, there exists $t_0\in X$ such that 
$\phi^{\dist_X^p}(s_0)\leq -\dist_X(t_0, s_0)^p-\phi(t_0)+\varepsilon$ thus
\begin{align}\label{cpconti}
\begin{split}
    \phi^{\dist_X^p}(s_0)-\phi^{\dist_X^p}(s_j)
    &\leq -\dist_X(t_0, s_0)^p+\dist_X(t_0, s_j)^p+\varepsilon\\
    &\leq p\cdot\max\{\dist_X(t_0, s_j)^{p-1}, \dist_X(t_0, s_0)^{p-1}\}\lvert\dist_X(t_0, s_j)-\dist_X(t_0, s_0)\rvert+\varepsilon\\
    &\leq p\cdot\max\{\dist_X(t_0, s_j)^{p-1}, \dist_X(t_0, s_0)^{p-1}\}\dist_X(s_j, s_0)+\varepsilon\\
    &< 2\varepsilon
\end{split}
\end{align}
if $j$ is sufficiently large. Similarly, for any $j\in \N$,
we have 
\begin{align}\label{cpcontii}
    \phi^{\dist_X^p}(s_j)-\phi^{\dist_X^p}(s_0)
    &\leq p\max\{\dist_X(t_j, s_j)^{p-1}, \dist_X(t_j, s_0)^{p-1}\}\dist_X(s_j, s_0)+\varepsilon,
\end{align}
where $t_j\in X$ satisfies  $\phi^{\dist_X^p}(s_j)\leq -\dist_X(t_j, s_j)^p-\phi(t_j)+\varepsilon$.
We also find that 
\begin{align*}
    \dist_X(t_j, s_0)^p&\leq 2^{p-1}(\dist_X(s_0, s_j)^p+\dist_X(t_j, s_j)^p)\\
    &\leq 2^{p-1}(\dist_X(s_0, s_j)^p-\phi^{\dist_X^p}(s_j)-\phi(t_j)+\varepsilon)
    \leq 2^{p-1}(\dist_X(s_0, s_j)^p+2\lVert\phi\rVert_{C_b(X)}+\varepsilon),
\end{align*}
which implies that  $\{t_j\}_{j\in\N}$ is bounded, consequently for $j$ sufficiently large,
\begin{align*}
    \phi^{\dist_X^p}(s_j)-\phi^{\dist_X^p}(s_0)
    &\leq p\max\{\dist_X(t_j, s_j)^{p-1}, \dist_X(t_j, s_0)^{p-1}\}\dist_X(s_j, s_0)+\varepsilon<2\varepsilon.
\end{align*}
Thus we see $\phi^{\dist_X^p}$ is continuous.

Now by definition, 
\[
-\dist_X(t, s)^p\leq \psi(t)+\psi^{\dist_X^p}(s) 
\quad \text{for\ }t, s\in X,
\]
hence
\begin{align*}
 \phi^{\dist_X^p}(s)-\psi^{\dist_X^p}(s)
&=\sup_{t\in X} \left(-\dist_X(t, s)^p-\phi(t)\right)-\psi^{\dist_X^p}(s)\\
&\leq \sup_{t\in X} \left(\psi(t)+\psi^{\dist_X^p}(s)-\phi(t)\right)-\psi^{\dist_X^p}(s)\\
&=\|\phi-\psi\|_{C_b(X)},
\end{align*}
 and switching the roles of $\phi$ and $\psi$
completes the proof of the lemma.
\end{proof}
%%%%%%
For the remainder of this section, we write $\dist_\R$ for the Euclidean distance on $\R$.
%%%%%%%
We now recall the classical duality for $\mk_p^\R$, also known as \emph{Kantorovich duality}, which will be the basis of a duality theory for $\mk_{p, q}$.
%%%%
\begin{theorem}[\cite{Villani09}*{Theorem~5.10}]\label{Kantorovich}
Let $1\leq p<\infty$, then for $\mu,\nu\in \mathcal{P}(\R)$,  
\begin{align*}
\mk_p^\R(\mu,\nu)^p
&=\sup
\left\{
-\int_\R\phi d\mu-\int_\R\psi d\nu
\Biggm|
\begin{tabular}{ll}
$(\phi,\psi)\in C_b(\R)^2$,\\
with $-\phi(t)-\psi(s) \leq \lvert t-s\rvert^p$ 
for $(t,s)\in \R^2$
\end{tabular}
\right\}\\
&=\sup
\left\{
-\int_\R\phi d\mu-\int_\R\phi^{\dist_\R^p} d\nu
\Bigm|
\phi\in C_b(\R)
\right\}.
\end{align*}
\end{theorem}
The duality for $\mk_{p, q}(\mu, \nu)$ comes about in the natural way, by applying the classical Kantorovich duality to the pair $\rad^\omega_\sharp\mu$ and $\rad^\omega_\sharp\nu$ for each $\omega\in \S^{n-1}$, then ``gathering'' the corresponding problems. However, we must be careful of the dependencies on the variable~ $\omega$ that arise.
\begin{definition}\label{FFF}
For $\mu$, $\nu\in \mathcal{P}_p(\R^n)$ with $1\leq p<\infty$
and $\varepsilon>0$,
we define a set-valued function $F^{\mu,\nu}_\varepsilon$ from $\mathbb{S}^{n-1}$ to $2^{C_b(\mathbb{R})}$ by
\begin{align*}
F^{\mu,\nu}_\varepsilon(\omega)
&:=\left\{
\phi \in C_b(\mathbb{R}) \Biggm|
 -\int_{\R}\phi d\rad^\omega_\sharp\mu-\int_{\R}\phi^{\dist_{\mathbb{R}}^p} d\rad^\omega_\sharp\nu
 > \mk_p^{\R}(\rad^\omega_\sharp\mu, \rad^\omega_\sharp\nu)^p-\varepsilon
\right\}.
\end{align*}
\end{definition}
Let us recall the following approximate selection property due to Michael.
%%%%
\begin{proposition}[\cite{Michael56}*{Lemma~1}]\label{Michael56}
Let $\Omega$ be a paracompact space, and $(X,\|\cdot\|_X)$ a normed space.
For a map $F:\Omega\to 2^X$,
if $F(\omega)$ is nonempty and convex for any $\omega\in \Omega$,
and $F$ is lower semi-continuous, that is, 
$\left\{\omega\in \Omega \mid F(\omega) \cap O\right\}$ is open in $\Omega$ for any open set $O$ in $X$, 
%%%%
then, for any $r>0$,
there exists a continuous map $f:\Omega \to X$ such that 
\[
\inf_{x\in F(\omega)} \|x-f(\omega)\|_X<r
\]
for every $\omega\in \Omega$.
\end{proposition}

We will need to apply this approximate selection result to $F^{\mu, \nu}_\varepsilon$.
\begin{lemma}\label{lem: lsc}
Let $\mu$, $\nu\in \mathcal{P}_p(\R^n)$ and $\varepsilon>0$.
Then $F^{\mu,\nu}_\varepsilon (\omega)$ is convex and nonempty for each $\omega \in \mathbb{S}^{n-1}$,
and $F^{\mu,\nu}_\varepsilon$ is lower semi-continuous.
%%%
\end{lemma}
\begin{proof}
Since $\mu$, $\nu\in \mathcal{P}_p(\R^n)$ and $\varepsilon>0$ are fixed, we write $F$ in place of $F^{\mu,\nu}_\varepsilon$. It is clear from Theorem~\ref{Kantorovich} that $F(\omega)\neq \emptyset$ for all $\omega\in \S^{n-1}$.

Next we show that $F(\omega)$ is convex for each $\omega \in \mathbb{S}^{n-1}$. 
Let $\phi_0$, $\phi_1\in F(\omega)$, and $\tau\in (0, 1)$. 
Then, for any $s\in \R$, we find that 
\begin{align*}
\left[ (1-\tau)\phi_0+\tau\phi_1 \right]^{\dist_\R^p}(s)
&
=\sup_{t\in \R}
\left[
(1-\tau)\left(-\dist_\R(t, s)^p-\phi_0(t)\right)
+\tau\left(-\dist_\R(t, s)^p-\phi_1(t)\right)\right]\\
%%%
&\leq \sup_{t\in \R}
\left[
(1-\tau)(-\dist_\R(t, s)^p-\phi_0(t))\right]
+\sup_{t\in \R}
\left[
\tau(-\dist_\R(t, s)^p-\phi_1(t))\right]\\
%%%
&=(1-\tau)\phi_0^{\dist_\R^p}(s)+\tau \phi_1^{\dist_\R^p}(s).
\end{align*}
Hence, 
\begin{align*}
&-\int_{\R}\left[(1-\tau)\phi_0+\tau\phi_1\right)]d\rad^\omega_\sharp\mu
-\int_{\R}\left[(1-\tau)\phi_0+\tau\phi_1\right]^{\dist_\R^p} d\rad^\omega_\sharp\nu\\
%%%
&\geq (1-\tau)\left(-\int_{\R}\phi_0d\rad^\omega_\sharp\mu-\int_{\R}\phi_0^{\dist_\R^p}d\rad^\omega_\sharp\nu\right)+\tau\left(-\int_{\R}\phi_1d\rad^\omega_\sharp\mu-\int_{\R}\phi_1^{\dist_\R^p}d\rad^\omega_\sharp\nu\right)\\
%%%%%%
&> \mk_p^{\R}(\rad^\omega_\sharp\mu, \rad^\omega_\sharp\nu)^p-\varepsilon,
\end{align*}
proving the convexity of $F(\omega)$.

Finally, we prove lower semi-continuity of $F$. 
Given $\phi\in C_b(\R)$
define $G_\phi: \S^{n-1}\to [-\infty,\infty)$ by
\begin{align*}
 G_\phi(\omega):=-\int_{\R}\phi d\rad^\omega_\sharp\mu-\int_{\R}\phi^{\dist_\R^p} d\rad^\omega_\sharp\nu-\mk_p^{\R}(\rad^\omega_\sharp\mu, \rad^\omega_\sharp\nu)^p.
\end{align*}
We see that
$\phi\in F(\omega)$ if and only if $G_\phi(\omega)>-\varepsilon$,
 hence if $G_\phi$ is continuous on $\mathbb{S}^{n-1}$ 
and $O\subset C_b(\R)$ is open, 
then we have
\begin{align*}
 \{ \omega\in \S^{n-1}\ |\ F(\omega) \cap O \neq \emptyset \}&=\bigcup_{\phi\in O}G_\phi^{-1}((-\varepsilon, \infty))
\end{align*}
which is then open as a union of open sets. 
In particular, this would prove $F$ is lower semi-continuous, thus the rest of the proof is devoted to showing the continuity of $G_\phi$.

Let $(\omega_j)_{j\in \N}$ be a sequence in $\mathbb{S}^{n-1}$ converging to $\omega$. 
For $\phi\in C_b(\R)$, by Lemma~\ref{cpp} we have that $\phi^{\dist_\R^p}\in C_b(\mathbb{R})$.
We observe from Lemma~\ref{prelemma}~\eqref{conti} together with Theorem~\ref{thm: wassconv} that 
\[
 \lim_{j\to\infty}
 \left(-\int_{\R}\phi d\rad^{\omega_j}_\sharp\mu-\int_{\R}\phi^{\dist_\R^p} d\rad^{\omega_j}_\sharp\nu\right)
 =-\int_{\R}\phi d\rad^\omega_\sharp\mu-\int_{\R}\phi^{\dist_\R^p} d\rad^\omega_\sharp\nu. 
\]
This with Lemma~\ref{prelemma} \eqref{twocont} yields 
\begin{align*}
 \lim_{j\to\infty}G_\phi(\omega_j)=G_\phi(\omega),
\end{align*}
hence $G_\phi$ is continuous.
\end{proof}
%%%%%%%%%
%%%%%%%%%
Now we are in position to prove the claimed duality result.
\begin{proof}[Proof of \nameref{thm: main sliced}~\eqref{thm: sliced duality}]
Let $(\Phi, \Psi)\in \mathcal{A}_p$, and $\zeta\in \mathcal{Z}_{r'}$; recall that we now assume $p\leq q$ and write $r:=q/p$.
By the dual representation for $L^r$-norms (see \cite{Folland99}*{Proposition 6.13}, this is applicable for the case $p=q$ as well since $\sigma_{n-1}$ is finite), and the Kantorovich duality \nameref{Kantorovich},
\begin{align*}
\mk_{p, q}(\mu, \nu)^p
&=\lVert\mk_p^{\R}(\rad^\bullet_\sharp\mu, \rad^\bullet_\sharp\nu)^p\rVert_{L^{r}(\sigma_{n-1})}\\
&\geq
\int_{\S^{n-1}}\zeta
\left(
-\int_{\mathbb{R}} \Phi_\bullet(t) d\rad^\bullet_\sharp\mu(t)
-\int_{\mathbb{R}} \Psi_\bullet(s) d\rad^\bullet_\sharp\nu(s)
\right)d\sigma_{n-1}.
\end{align*}

Next we show the reverse inequality.
Since the case of $\mu=\nu$ is trivial, we assume $\mu \neq \nu$.
Fix $\varepsilon>0$, then by Lemma~\ref{lem: lsc}, we can apply Proposition~\ref{Michael56} to find $f_{\bullet}\in C( \S^{n-1}; C_b(\R) )$ such that
\[
\inf_{\phi\in F^{\mu,\nu}_\varepsilon(\omega)}\lVert \phi-f_\omega\rVert_{C_b(\R)}<\varepsilon
\quad
\text{for all\ }\omega\in \S^{n-1},
\]
in particular for each $\omega\in \S^{n-1}$, there exists $\phi^\omega \in F^{\mu,\nu}_\varepsilon(\omega)$ 
such that $\lVert \phi^\omega-f_\omega\rVert_{C_b(\R)}<\varepsilon$. 
This with Lemma~\ref{cpp} yields
$\lVert (\phi^\omega)^{\dist_\R^p}-f_\omega^{\dist_\R^p}\rVert_{C_b(\R)}<\varepsilon$.
From Lemma~\ref{cpp} we also find $f_\bullet^{\dist_\R^p}\in C(\mathbb{S}^{n-1};C_b(\mathbb{R}))$
hence $(f_\bullet, f_\bullet^{\dist_\R^p})\in \mathcal{A}_p$.
Then we calculate
\begin{align}
\begin{split} \label{eqn: selection estimate}
&-\int_{\R^n}f_\omega(\langle x,\omega\rangle)d\mu(x)
-\int_{\R^n}
f_\omega^{\dist_\R^p}(\langle y,\omega\rangle)d\nu(y) \\
%%%%
&\geq-\int_{\R^n}\phi^\omega(\langle x, \omega\rangle)d\mu(x)
-\int_{\R^n}(\phi^\omega)^{\dist_\R^p}(\langle y,\omega\rangle)d\nu(y)
-2\varepsilon\\
%%%%% 
&=-\int_{\R}\phi^\omega d\rad^\omega_\sharp\mu 
-\int_{\R}(\phi^\omega)^{\dist_\R^p} d\rad^\omega_\sharp\nu-2\varepsilon\\
%%%% 
&\geq \mk_p^{\R}(\rad^\omega_\sharp\mu, \rad^\omega_\sharp\nu)^p-3\varepsilon,
\end{split} 
\end{align}
where the last inequality follows from $\phi^\omega \in F^{\mu,\nu}_\varepsilon(\omega)$.
%%%%%
Since $\mk_{p, q}(\mu, \nu)<\infty$, then again by dual representation for the $L^{r}$-norm (\cite{Folland99}*{Proposition 6.13}), 
there exists some $\tilde\zeta\in L^{r'}(\sigma_{n-1})$ with $\lVert \tilde\zeta\rVert_{L^{r'}(\sigma_{n-1})}<1$ such that
\begin{align}\label{eqn: Lq/p norm dual estimate 1 sliced}
 \int_{\S^{n-1}}
 \tilde\zeta\mk_p^\R(\rad^\bullet_\sharp\mu, \rad^\bullet_\sharp\nu)^p d\sigma_{n-1}
 &> \lVert\mk_p^\R(\rad^\bullet_\sharp\mu, \rad^\bullet_\sharp\nu)^p\rVert_{L^{r}(\sigma_{n-1})}-\varepsilon=\mk_{p, q}(\mu, \nu)^p-\varepsilon,
\end{align}
since $\mk_p^\R(\rad^\bullet_\sharp\mu, \rad^\bullet_\sharp\nu)^p\geq 0$ we may assume $\tilde{\zeta}\geq 0$ holds $\sigma_{n-1}$-a.e..
If $p<q$ we have $r'<\infty$, 
then by density of $C_b(\S^{n-1})$ in $L^{r'}(\sigma_{n-1})$,
we can find $\zeta\in C_b(\S^{n-1})$ such that
\[
\lVert \zeta-\tilde\zeta\rVert_{L^{r'}(\sigma_{n-1})}<\mk_{p, q}(\mu, \nu)^{-p}\varepsilon, \qquad
\lVert \zeta\rVert_{L^{r'}(\sigma_{n-1})}< 1,
\]
and we may assume $\zeta\geq 0$ everywhere on $\mathbb{S}^{n-1}$.
If $p=q$ we have $r'=\infty$, 
thus by positivity of $\mk_p^\R(\rad^\omega_\sharp\mu, \rad^\omega_\sharp\nu)^p$ 
we can choose $\zeta=\tilde{\zeta}\equiv 1$ on $\mathbb{S}^{n-1}$.
%%%
Then using H\"older's inequality, 
\begin{align*}
\int_{\S^{n-1}}
\left|(\zeta- \tilde\zeta) \mk_p^\R(\rad^\bullet_\sharp\mu, \rad^\bullet_\sharp\nu)^p
 \right| d\sigma_{n-1}
\leq 
\lVert \zeta-\tilde\zeta\rVert_{L^{r'}(\sigma_{n-1})}
\|\mk_p^\R(\rad^\bullet_\sharp\mu, \rad^\bullet_\sharp\nu)^p\|_{L^r(\sigma_{n-1})}
<\varepsilon,
\end{align*}
which yields by~\eqref{eqn: Lq/p norm dual estimate 1 sliced},
\begin{align}\label{eqn: Lq/p norm dual estimate 2 sliced}
 \int_{\S^{n-1}}\zeta\mk_p^\R(\rad^\bullet_\sharp\mu, \rad^\bullet_\sharp\nu)^pd\sigma_{n-1}&> \mk_{p, q}(\mu, \nu)^p-2\varepsilon,
\end{align}
since $\mk_p^\R(\rad^\omega_\sharp\mu, \rad^\omega_\sharp\nu)\geq 0$ we may add a small constant to $\zeta$ to assume that $\zeta>0$ on $\S^{n-1}$ with~\eqref{eqn: Lq/p norm dual estimate 2 sliced} remaining valid. 

Multiplying \eqref{eqn: selection estimate} by $\zeta$ then integrating, we find
\begin{align*}
&\int_{\S^{n-1}}\zeta(\omega)
\left(-\int_{\R}f_\omega d\rad^\omega_\sharp\mu
-\int_{\R}
f_\omega^{\dist_\R^p}d\rad^\omega_\sharp\nu\right)d\sigma_{n-1}(\omega)\\
&=\int_{\S^{n-1}}\zeta(\omega)
\left(-\int_{\R^n}f_\omega(\langle x,\omega\rangle)d\mu(x)
-\int_{\R^n}
f_\omega^{\dist_\R^p}(\langle y,\omega\rangle)d\nu(y)\right)d\sigma_{n-1}(\omega)\\
 &\geq \int_{\S^{n-1}}\zeta(\omega)\mk_p^{\R}(\rad^\omega_\sharp\mu, \rad^\omega_\sharp\nu)^pd\sigma_{n-1}(\omega)-3\varepsilon\\
 &=\mk_{p, q}(\mu, \nu)^p-5\varepsilon,
\end{align*}
where in the third line we have used that $\lVert \zeta\rVert_{L^1(\sigma_{n-1})}\leq \lVert \zeta\rVert_{L^{r'}(\sigma_{n-1})}\leq 1$. Since $\varepsilon>0$ is arbitrary, this finishes the proof.
\end{proof}
%%%%%%%%%%%%%%%%%%%%%%%
\begin{ack} 
The authors would like to thank Guillaume Carlier, Wilfrid Gangbo, Quentin M{\'e}rigot, Brendan Pass, and Dejan Slep{\v c}ev for fruitful discussions,  Felix Otto for pointing out the geodesic nature of the sliced metrics when $p=1$, and March Boedihardjo for pointing out that $\mk_{1, \infty}$ is not bi-Lipschitz equivalent to $\mk_1^{\R^n}$. The authors would also like to thank Kota Matsui for bringing this topic to their attention.

JK was supported in part by National Science Foundation grants DMS-2000128 and DMS-2246606.

AT was supported in part by JSPS KAKENHI Grant Number 19H01786, 19K03494.
\end{ack}

\medskip

%%%%%%%%%%%%%%%%%
\bibliography{Wpq.bib}

\def\cprime{$'$}
\begin{bibdiv}
\begin{biblist}

\bib{AguehCarlier11}{article}{
      author={Agueh, Martial},
      author={Carlier, Guillaume},
       title={Barycenters in the {W}asserstein space},
        date={2011},
        ISSN={0036-1410},
     journal={SIAM J. Math. Anal.},
      volume={43},
      number={2},
       pages={904\ndash 924},
         url={https://doi-org.proxy2.cl.msu.edu/10.1137/100805741},
      review={\MR{2801182}},
}

\bib{AmbrosioGigliSavare08}{book}{
      author={Ambrosio, Luigi},
      author={Gigli, Nicola},
      author={Savar{\'e}, Giuseppe},
       title={Gradient flows in metric spaces and in the space of probability
  measures},
     edition={Second},
      series={Lectures in Mathematics ETH Z\"urich},
   publisher={Birkh\"auser Verlag, Basel},
        date={2008},
        ISBN={978-3-7643-8721-1},
      review={\MR{2401600 (2009h:49002)}},
}

\bib{BayraktarGuo21}{article}{
      author={Bayraktar, Erhan},
      author={Guo, Gaoyue},
       title={Strong equivalence between metrics of {W}asserstein type},
        date={2021},
     journal={Electron. Commun. Probab.},
      volume={26},
       pages={Paper No. 13, 13},
         url={https://doi-org.proxy2.cl.msu.edu/10.3390/mca26010013},
      review={\MR{4236683}},
}

\bib{BobkovLedoux19}{article}{
      author={Bobkov, Sergey},
      author={Ledoux, Michel},
       title={One-dimensional empirical measures, order statistics, and
  {K}antorovich transport distances},
        date={2019},
        ISSN={0065-9266,1947-6221},
     journal={Mem. Amer. Math. Soc.},
      volume={261},
      number={1259},
       pages={v+126},
         url={https://doi.org/10.1090/memo/1259},
      review={\MR{4028181}},
}

\bib{Bogachev07}{book}{
      author={Bogachev, V.~I.},
       title={Measure theory. {V}ol. {I}, {II}},
   publisher={Springer-Verlag, Berlin},
        date={2007},
        ISBN={978-3-540-34513-8; 3-540-34513-2},
         url={https://doi-org.proxy2.cl.msu.edu/10.1007/978-3-540-34514-5},
      review={\MR{2267655}},
}

\bib{BonneelRabinPeyrePfister15}{article}{
      author={Bonneel, Nicolas},
      author={Rabin, Julien},
      author={Peyr\'{e}, Gabriel},
      author={Pfister, Hanspeter},
       title={Sliced and {R}adon {W}asserstein barycenters of measures},
        date={2015},
        ISSN={0924-9907,1573-7683},
     journal={J. Math. Imaging Vision},
      volume={51},
      number={1},
       pages={22\ndash 45},
         url={https://doi.org/10.1007/s10851-014-0506-3},
      review={\MR{3300482}},
}

\bib{Bonnotte}{article}{
      author={Bonnotte, Nicolas},
       title={Unidimensional and evolution methods for optimal transportation},
        date={2013},
     journal={Ph.D. thesis},
        note={\url{http://cvgmt.sns.it/paper/2341/}},
}

\bib{mSWD}{inproceedings}{
      author={Deshpande, I.},
      author={Hu, Y.},
      author={Sun, R.},
      author={Pyrros, A.},
      author={Siddiqui, N.},
      author={Koyejo, S.},
      author={Zhao, Z.},
      author={Forsyth, D.},
      author={Schwing, A.~G.},
       title={Max-sliced wasserstein distance and its use for gans},
        date={2019},
   booktitle={2019 {IEEE/CVF} {C}onference on {C}omputer {V}ision and {P}attern
  {R}ecognition {(CVPR)}},
   publisher={IEEE Computer Society},
     address={Los Alamitos, CA, USA},
       pages={10640\ndash 10648},
         url={https://doi.ieeecomputersociety.org/10.1109/CVPR.2019.01090},
}

\bib{Folland99}{book}{
      author={Folland, Gerald~B.},
       title={Real analysis},
     edition={Second},
      series={Pure and Applied Mathematics (New York)},
   publisher={John Wiley \& Sons, Inc., New York},
        date={1999},
        ISBN={0-471-31716-0},
        note={Modern techniques and their applications, A Wiley-Interscience
  Publication},
      review={\MR{1681462}},
}

\bib{Galichon16}{book}{
      author={Galichon, Alfred},
       title={Optimal transport methods in economics},
   publisher={Princeton University Press, Princeton, NJ},
        date={2016},
        ISBN={978-0-691-17276-7},
         url={https://doi.org/10.1515/9781400883592},
      review={\MR{3586373}},
}

\bib{Hertle83}{article}{
      author={Hertle, Alexander},
       title={Continuity of the {R}adon transform and its inverse on
  {E}uclidean space},
        date={1983},
        ISSN={0025-5874,1432-1823},
     journal={Math. Z.},
      volume={184},
      number={2},
       pages={165\ndash 192},
         url={https://doi.org/10.1007/BF01252856},
      review={\MR{716270}},
}

\bib{Hertle84}{article}{
      author={Hertle, Alexander},
       title={On the range of the {R}adon transform and its dual},
        date={1984},
        ISSN={0025-5831,1432-1807},
     journal={Math. Ann.},
      volume={267},
      number={1},
       pages={91\ndash 99},
         url={https://doi.org/10.1007/BF01458472},
      review={\MR{737337}},
}

\bib{Ledoux17}{article}{
      author={Ledoux, M.},
       title={On optimal matching of {G}aussian samples},
        date={2017},
        ISSN={0373-2703},
     journal={Zap. Nauchn. Sem. S.-Peterburg. Otdel. Mat. Inst. Steklov.
  (POMI)},
      volume={457},
       pages={226\ndash 264},
      review={\MR{3723584}},
}

\bib{Michael56}{article}{
      author={Michael, E.},
       title={Selected {S}election {T}heorems},
        date={1956},
        ISSN={0002-9890},
     journal={Amer. Math. Monthly},
      volume={63},
      number={4},
       pages={233\ndash 238},
         url={https://doi-org.proxy2.cl.msu.edu/10.2307/2310346},
      review={\MR{1529282}},
}

\bib{ParkSlepcev23}{article}{
      author={Park, Sangmin},
      author={Slep{\v c}ev, Dejan},
       title={Geometry and analytic properties of the sliced {W}asserstein
  space},
        date={2023},
     journal={arXiv},
      eprint={2311.05134},
}

\bib{CuturiPaty19}{inproceedings}{
      author={Paty, Fran{\c{c}}ois-Pierre},
      author={Cuturi, Marco},
       title={Subspace robust {W}asserstein distances},
        date={201909--15 Jun},
   booktitle={Proceedings of the 36th international conference on machine
  learning},
      editor={Chaudhuri, Kamalika},
      editor={Salakhutdinov, Ruslan},
      series={Proceedings of Machine Learning Research},
      volume={97},
   publisher={PMLR},
       pages={5072\ndash 5081},
         url={https://proceedings.mlr.press/v97/paty19a.html},
}

\bib{sliced-original}{inproceedings}{
      author={Rabin, Julien},
      author={Peyr{\'e}, Gabriel},
      author={Delon, Julie},
      author={Bernot, Marc},
       title={Wasserstein barycenter and its application to texture mixing},
        date={2012},
   booktitle={Scale space and variational methods in computer vision},
      editor={Bruckstein, Alfred~M.},
      editor={ter Haar~Romeny, Bart~M.},
      editor={Bronstein, Alexander~M.},
      editor={Bronstein, Michael~M.},
   publisher={Springer Berlin Heidelberg},
     address={Berlin, Heidelberg},
       pages={435\ndash 446},
}

\bib{Santambrogio15}{book}{
      author={Santambrogio, Filippo},
       title={Optimal transport for applied mathematicians},
      series={Progress in Nonlinear Differential Equations and their
  Applications},
   publisher={Birkh\"{a}user/Springer, Cham},
        date={2015},
      volume={87},
        ISBN={978-3-319-20827-5; 978-3-319-20828-2},
         url={https://doi-org.proxy2.cl.msu.edu/10.1007/978-3-319-20828-2},
        note={Calculus of variations, PDEs, and modeling},
      review={\MR{3409718}},
}

\bib{Stroock94}{book}{
      author={Stroock, Daniel~W.},
       title={A concise introduction to the theory of integration},
     edition={Second},
   publisher={Birkh\"{a}user Boston, Inc., Boston, MA},
        date={1994},
        ISBN={0-8176-3759-1},
         url={https://doi-org.proxy2.cl.msu.edu/10.1007/978-1-4757-2300-7},
      review={\MR{1267228}},
}

\bib{Villani09}{book}{
      author={Villani, C{\'e}dric},
       title={Optimal transport: Old and new},
      series={Grundlehren der Mathematischen Wissenschaften [Fundamental
  Principles of Mathematical Sciences]},
   publisher={Springer-Verlag},
     address={Berlin},
        date={2009},
      volume={338},
        ISBN={978-3-540-71049-3},
         url={http://dx.doi.org/10.1007/978-3-540-71050-9},
      review={\MR{2459454 (2010f:49001)}},
}

\bib{xu2022central}{misc}{
      author={Xu, Xianliang},
      author={Huang, Zhongyi},
       title={Central limit theorem for the sliced 1-wasserstein distance and
  the max-sliced 1-wasserstein distance},
        date={2022},
}

\end{biblist}
\end{bibdiv}
\bibliographystyle{alpha}
\end{document}